\documentclass[11pt]{amsart}
\usepackage{amsmath}
\usepackage{a4wide}
\usepackage[utf8]{inputenc}
\usepackage{amssymb}
\usepackage{amsopn}
\usepackage{epsfig}
\usepackage{amsfonts}
\usepackage{latexsym}
\usepackage{amsthm}
\usepackage{enumerate}
\usepackage[UKenglish]{babel}
\usepackage{bm}
\usepackage{verbatim}
\usepackage{color}
\usepackage{pgf}
\usepackage{tikz}

\usepackage{mathrsfs}   
 \usepackage{mathtools}
 \usepackage{graphicx}   

\usetikzlibrary{arrows,automata}

\newtheorem{theorem}{Theorem}[section]
\newtheorem{lemma}[theorem]{Lemma}
\newtheorem{proposition}[theorem]{Proposition}

\theoremstyle{definition}

\newtheorem{example}[theorem]{Example}

\theoremstyle{remark}
\newtheorem{remark}[theorem]{Remark}

\numberwithin{equation}{section}

\newcommand{\R}{\ensuremath{\mathbb{R}}}

\newcommand{\N}{\ensuremath{\mathbb{N}}}

\newcommand{\sd}{\mathbf{d}}
\renewcommand{\sc}{\mathbf{c}}

\newcommand{\set}[1]{\left\{#1\right\}}
\newcommand{\la}{\lambda}

\newcommand{\ep}{\varepsilon}
\newcommand{\f}{\infty}

\newcommand{\om}{\omega}

\newcommand{\al}{\alpha}

\newcommand{\si}{\sigma}

\newcommand{\diam}{\operatorname{diam}}
\begin{document}

\title{Open dynamical  systems with a moving  hole}

\author[D. Kong]{Derong Kong}
\address[D. Kong]{College of Mathematics and Statistics, Center of Mathematics, Chongqing University, Chongqing, 401331, P.R.China}
\email{derongkong@126.com}

\author[B. Sun]{Beibei Sun}
\address[B. Sun]{College of Mathematics and Statistics,   Chongqing University, Chongqing, 401331, P.R.China}
\email{beibeisun9@163.com}

\author[Z. Wang]{Zhiqiang Wang}
\address[Z. Wang]{College of Mathematics and Statistics, Center of Mathematics, Key Laboratory of Nonlinear Analysis and its Applications (Ministry of Education), Chongqing University, Chongqing 401331, P.R.China}
\email{zhiqiangwzy@163.com,~zqwangmath@cqu.edu.cn}

\dedicatory{}



\begin{abstract}
  Given an integer $b\ge 3$, let  $T_b: [0,1)\to [0,1); x\mapsto bx\pmod 1$ be the expanding map on the unit circle.  For any $m\in\N$ and for any $\om=\om^0\om^1\ldots\in(\set{0,1,\ldots,b-1}^m)^{\N_0}$ let
 \[
   K^\om=\set{x\in[0,1): T_b^n(x)\notin I_{\om^n}~\forall n\ge 0},
 \]
 where $I_{\om^n}$ is the $b$-adic basic interval generated by $\om^n$. Then $K^\om$ is called the survivor set of the open dynamical system $([0,1), T_b, I_\om)$ with respect to the sequence of  holes $I_\om=\set{I_{\om^n}: n\ge 0}$. We show that the Hausdorff and lower box dimensions of $K^\om$ always coincide, and the packing and upper box dimensions of $K^\om$ also coincide. Moreover, we give sharp lower  and upper bounds for the dimensions of $K^\om$, which can be calculated explicitly. Finally, for any admissible $\al\le \beta$ we show that there exist infinitely many $\om$ (in fact of positive dimension) such that $\dim_H K^\om=\alpha$ and $\dim_P K^\om=\beta$.

 As applications we first study badly approximable numbers in Diophantine approximation.  For an arbitrary sequence of balls  $\set{B_n}$ on $[0,1)$, let $K\left(\set{B_n}\right)$ be the set of $x\in[0,1)$ such that  $T_b^n(x)\notin B_n$ for all but finitely many $n\ge 0$.
  Assuming   the limit $\lim_{n\to\f}\diam \left(B_n\right)$ exists, we show that  $\dim_H K\left(\set{B_n}\right)=1$ if and only if $\lim_{n\to\f}\diam  \left(B_n\right)=0$. Similarly, our results can be applied to  the set of non-recurrence points. For any positive function $\phi$  on $\mathbb{N}$, let $E\left(\phi\right)$ be the set of $x\in[0,1)$ satisfying  $|T_b^n (x)-x|\ge \phi(n)$ for all but finitely many $n$. Then we prove, under the existence of the limit $\lim_{n\to\f}\phi(n)$, that $\dim_H E(\phi)=1$ if and only if $\lim_{n\to\f}\phi(n)=0$.   Our results can also be applied to study  joint spectral radius of matrices.  We show that the finiteness property for the joint spectral radius of associated adjacency  matrices holds true.
\end{abstract}


\maketitle

\section{Introduction}

The mathematical study of dynamical systems with holes, called \emph{open dynamical systems}, was first proposed by Pianigiani and Yorke \cite{Pianigiani-Yorke-1979}. Let $(X, T, \mu)$ be an ergodic dynamical system on    a compact metric space $X$, and let $H\subset X$ be the \emph{hole}  with $\mu(H)>0$. For $n\in\N$ let $P_s(n)$ be the survivor probability defined by $P_s(n):=\mu\left(\set{x\in X: T^j (x)\notin H~\forall 0\le j\le n}\right)$.  In open dynamical systems people always study the decay rate of $P_s(n)$, called the \emph{escape rate}, given by $\rho(H):=\lim_{n\to\f}-\frac{1}{n}\log P_s(n)$. For more information on the escape rate and the related first hitting time we refer to the papers \cite{Bruin-Demers-Todd-2018,  Demers-2005, Demers-Wright-Young-2012, Demers-Young-2006} and the references therein. Note by Birkhoff ergodic theorem (cf.~\cite{Walters-82}) that the survivor set
\[
K(H):=\set{x\in X: T^n (x)\notin H~\forall n\in \mathbb{N}_0}
\]
is a $\mu$-null set. Then the fractal dimension, such as Hausdorff dimension, is applied   to estimate the size of $K(H)$. It turns out that the dimension of $K(H)$ depends not only on the size  but also on the position of $H$ (cf.~\cite{Bunimovich-Yurchenko-2011}). For more results on the dimension of $K(H)$ we refer to the papers \cite{Allaart-Kong-2023, Glendinning-Sidorov-2015, Kalle-Kong-Langeveld-Li-2020, Urbanski-1986}. Note that the study of   $K(H)$ has intimate  connection with unique $\beta$-expansions (cf.~\cite{AlcarazBarrera-Baker-Kong-2016, Allaart-Kong-19, deVries-Komornik-09, Glendinning_Sidorov_2001, Komornik-Kong-Li-2017, Sidorov_2007})  and  badly approximable numbers in Diophantine approximation (cf.~\cite{Dani-88, Farm-Persson-Schmeling-2010, Li-Liao-Velani-Zorin-23, Nielsen_1999}).

 Until now, in most work of open dynamical systems the hole $H$ was assumed to be fixed. In this paper  we will study the survivor set of an open dynamical system with a moving hole. In particular, the hole moves when the time $n$ changes. Our study  involves  a new type of fractals   including the Moran type fractals and graph-directed fractals as special cases. Our results can be applied to  the study of liminf sets in dynamical Diophantine approximations and to the finiteness property for the joint spectral radius of matrices.

\subsection{Open dynamical systems with a moving hole}\label{subsec:open-dynamics}
Let $\N:=\{1,2,3,\ldots\}$ be the set of all positive integers, and let $\N_0 := \N \cup \{0\}$.
For $k \in \N$ let $\N_{\ge k}:= \N \cap [k,+\f)$.

Given $b\in \mathbb{N}_{\geq  3}$, let $D_b:=\left\{0,1,\ldots,b -1\right\}$. For $n\in\N$ let $D_b^n$ be the set of all words of length $n$ over the alphabet $D_b$, i.e.,
 $D_b^n = \{d_1d_2 \ldots d_n : d_i \in D_b~\forall 1 \leq i \leq n\}.$
Let $D_b^*:=\bigcup_{n=1}^\f D_b^n$ be the set of all finite words over the alphabet $D_b$. Moreover,
let $D_b^\N :=\{(d_i)_{i=1}^\f: d_i \in D_b~\forall i \in \N\}$ denote the set of all infinite sequences over the alphabet $D_b$, and let $\sigma$ be the left shift map on $D_b^\N$. Then for $(d_i)_{i=1}^\f\in D_b^\N$ we have $\si((d_i)_{i=1}^\f)=(d_{i+1})_{i=1}^\f \in D_b^\N$. For a finite word $\sd=d_1d_2\ldots d_n\in D_b^*$, we define the cylinder set by $\mathcal{C}_\sd:=\{(\omega_i)_{i =1}^\f \in D_b^\N: \omega_i = d_i~ \forall 1 \le i \le n\}$.

Given $ m \in \mathbb{N} $, write $(D_b^m)^{\N_0}$ for the set of all infinite sequences $\om=\om^0\om^1\ldots$ with each $\om^k\in D_b^m$. For $\om=\om^0\om^1\ldots\in(D_b^m)^{\N_0}$ we first define the \emph{survivor set} of the symbolic dynamical system $(D_b^\N, \sigma)$ with the moving holes $\{\mathcal{C}_{\omega^k}: k \in \N_0\}$ by
\[ \Sigma^\om:=\big\{ (d_i)_{i=1}^\f \in D_b^\N: \sigma^k\big( (d_i)_{i=1}^\f \big) \not \in \mathcal{C}_{\omega^k}~\forall k \in \N_0 \big\}. \]
That is,
\begin{equation}\label{eq:survivor-set-symbolic}
  \Sigma^\om=\big\{ (d_i)_{i=1}^\f \in D_b^\N: d_{k+1}d_{k+2} \ldots d_{k+m} \ne \omega^k ~ \forall k \in \N_0 \big\}.
\end{equation}

Let $\pi_b: D_b^\N \to [0,1]$ be the projection map defined by $\pi_b\big( (d_i)_{i=1}^\f \big):= \sum_{i=1}^{\f} d_i b^{-i}$.
We are  interested in the survivor set
\begin{equation*}\label{eq:survivor-set-K-def}
  K^\om:=\pi_b(\Sigma^\om).
\end{equation*}
It is closely related to the survivor set of the open dynamical system $([0,1), T_b)$ with the moving holes $\big\{ \mathrm{int}\big( \pi_b(\mathcal{C}_{\omega^k}) \big): k \in \N_0\}$,  where $T_b: [0,1)\to [0,1);~ x\mapsto bx \pmod{1}$. Observe  that up to a countable set the map $\pi_b$ is  conjugate   from $(D_b^\N, \si)$ to $([0,1), T_b)$ (cf.~\cite{Lind_Marcus_1995}). So,
the survivor set
\begin{equation*}\label{eq:survivor-set-K-tilde}
  \widetilde{K}^\om :=\big\{ x \in [0,1): T_b^k(x) \not\in \mathrm{int}\big( \pi_b(\mathcal{C}_{\omega^k}) \big)~\forall k \in \N_0 \big\}
\end{equation*}
differs from $K^\om$ by at most countably many points. Since we are interested in the fractal dimension of $K^\om$, in the following we focus on  $K^\om$ which
 has a more explicit geometric construction
\begin{equation}\label{eq:survivor-set-K}
K^{\om}=\bigcap_{k=1}^\f \bigcup_{\sd \in \Sigma_k^{\om}}I_{\sd},
\end{equation}
where $I_\sd := \pi_b(\mathcal{C}_{\sd})$, and
\begin{equation}\label{eq:Sigma-k-om-def}
\Sigma_k^{\om}:=\big\{ d_1\ldots d_k: (d_i)_{i=1}^\f \in \Sigma^\om \big\}.
\end{equation}
Since $b\in \mathbb{N}_{\geq 3}$,  the survivor set $K^{\om}$ is a Cantor set, i.e., $K^\om$ is a compact set with neither interior nor isolated points.
Our work involves determining   various fractal dimensions of the survivor set $K^{\om}$.

For a bounded set $F\subset\R^n$, let $\dim_H F, ~\dim_P F$ denote its Hausdorff  and packing dimensions; let $\underline{\dim}_B F, ~\overline{\dim}_B F$ denote its lower and upper box dimensions (cf.~\cite{Falconer-14}); and let $\dim_L F,~\dim_A F$ denote its lower dimension and Assouad dimension (cf.~\cite{Fraser-2021}).
For a compact set $F\subset \R^n$, these fractal dimensions have the following relation (cf. \cite[Lemma 2.4.3, Theorem 3.4.3]{Fraser-2021})
\begin{equation}\label{eq:fractal-dimension}
  \dim_L F \leq \dim_H F \leq\min\set{\underline{\dim}_B F,\dim_P F}\le \max\set{\underline{\dim}_B F,\dim_P F} \le \overline{\dim}_B F\leq \dim_A F.
\end{equation}
In general, the lower box dimension and the packing dimension are incomparable.
If $F$ is $s$-Ahlfors regular, then all fractal dimensions in (\ref{eq:fractal-dimension}) are equal to $s$.
Recall that a compact set $F \subset \R^n$ is called \emph{$s$-Ahlfors regular} if there exists a constant $c \ge 1$ such that for all $0< r < 1$ and all $x \in F$,
\[ c^{-1} r^s \le \mathcal{H}^s\big( F \cap B(x,r) \big) \le c r^s,\]
where $\mathcal{H}^s$ is the $s$-dimensional Hausdorff measure, and $B(x,r)$ is the closed ball centered at $x$ with radius $r$.

When $m=1$, the set $K^\om$ is a homogeneous Moran set. Its Hausdorff  and packing dimensions were given by Feng et al.~\cite{Feng-Wen-Wu-97}, and its Assouad  and lower dimensions were given by Li et al.~\cite{Li-Li-Miao-Xi-2016} and Chen et al.~\cite{Chen-Wu-Wei-2017}, respectively.
In this case, all of these fractal dimensions of $K^\omega$ are equal to $\log(b-1)/ \log b$.
So in the following we always assume $m\ge 2$.

When $\omega^k = \sd$ for all $k \in \N_0$, we have $(\Sigma^\om,\sigma)$ is the shift of finite type with the forbidden word $\sd$, denoted by $X_\sd$ (cf.~\cite{Lind_Marcus_1995}).
It can be described by a directed graph $G_\sd=(V,E)$, where $V=D_b^{m-1}$, and for $\mathbf u=u_1\ldots u_{m-1}, \mathbf v=v_1\ldots v_{m-1}\in V$ we say $(\mathbf u,\mathbf v) \in E$ if $u_2\ldots u_{m-1}=v_1\ldots v_{m-2}$ and $u_1\ldots u_{m-1}v_{m-1}\ne \sd$.
Let $A_\sd$ be the corresponding adjacency matrix of $G_\sd$, whose size is $b^{m-1}\times b^{m-1}$.
In this case, all fractal dimensions in (\ref{eq:fractal-dimension}) of $K^\omega = \pi_b(X_\sd)$ are equal to $\log \rho(A_\sd) / \log b$, where $\rho(A)$ is the spectral radius of a matrix $A$ (cf.~\cite{Mauldin_Williams_1988}). To the best of our knowledge, $K^\om$ is a new type of fractal set for a general $\om=\om^0\om^1\ldots\in (D_b^m)^{\N_0}$.

For a real matrix $A=(a_{i,j})$ we use the norm $\| A  \|=\sum_{i,j}|a_{i,j}|$.
For a finite set $F$ let $|F|$ denote its cardinality.
In general, we obtain the following results concerning the fractal dimensions of $K^\om$.
\begin{theorem}\label{general-result}
For any $\om=\om^0\om^1\ldots\in (D_b^m)^{\mathbb{N}_0}$  we have
 $$\dim_H K^{\om} = \underline{\dim}_B K^{\om} = \liminf_{n \to \f} \frac{\log |\Sigma_n^{\om}|}{n \log b }=\liminf_{n\to\infty}\frac{1}{n \log b}\log \|A_{\om^0}A_{\om^1}\cdots A_{\om^{n-1}}\|,$$
 and
 $$\dim_P K^{\om} = \overline{\dim}_B K^{\om} = \limsup_{n \to \f} \frac{\log |\Sigma_n^{\om}|}{n \log b}=\limsup_{n\to\infty}\frac{1}{n \log b}\log\|A_{\om^0}A_{\om^1}\cdots A_{\om^{n-1}}\|.$$
\end{theorem}

\begin{remark}\label{rem:dim-equal}
  By a recent work of Fan and Wu \cite{Fan-Wu-2022} it follows that for any ergodic measure $\nu$ on $((D_b^m)^{\N_0}, \tilde\sigma)$ we have
  \[
  \dim_H K^\om=\dim_P K^\om=\lim_{n\to\f}\frac{1}{n\log b}\log \|A_{\om^0}A_{\om^1}\cdots A_{\om^{n-1}}\|\quad\textrm{for $\nu$-a.e. }\om\in (D_b^m)^{\N_0},
  \]
  where $\tilde\si(\om)=\om^1\om^2\ldots$ for $\om=\om^0\om^1\ldots\in (D_b^m)^{\N_0}$.
\end{remark}

To describe the range of  $\dim_H K^\om$ and $\dim_P K^\om$  we introduce two types of sequences  in $(D_b^m)^{\N_0}$, namely the progressively overlapping sequences and the totally distinct sequences.
Let $\om=\om^0\om^1\ldots\in (D_b^m)^{\mathbb{N}_0}$ with $\om^i= \om^i_{1} \ldots \om^i_m \in D_b^m$ for all $i \in \N_0$. For $k \in \N$, we say that the sequence $\om$ is \emph{progressively overlapping at position-$k$} if \[ \omega^k_1 \ldots \omega^k_{m-j} = \omega^{k-j}_{j+1} \ldots \omega^{k-j}_m\quad \forall 1 \le j \le \min\{k,m-1\}. \]
The sequence $\om$ is called \emph{progressively overlapping} if it is progressively overlapping at position-$k$ for all $k\in \mathbb{N}$, that is, $\omega^i_1 \ldots \omega^i_{m-1} = \omega^{i-1}_2 \ldots \omega^{i-1}_{m}$ for all $i \in \N$.
Each sequence $d_1 d_2 \ldots \in D_b^\N$ induces a progressively overlapping sequence $\om=\om^0\om^1\ldots\in (D_b^m)^{\mathbb{N}_0}$ by letting $\om^i=d_{i+1} \ldots d_{i+m}$ for all $i \in \N_0$.

In contrast with progressively overlapping sequences we define the totally distinct sequences.
We say that the sequence $\om$ is \emph{totally distinct at position-$k$} if \[ \omega^k_1 \ldots \omega^k_{m-j} \ne \omega^{k-j}_{j+1} \ldots \omega^{k-j}_m\quad \forall 1 \le j \le \min\{k,m-1\}. \]
The sequence $\om$ is called \emph{totally distinct} if it is totally distinct at position-$k$ for all $k\in \mathbb{N}$.

\begin{example}
  Let $b\in\N_{\ge 3}$ and $m=2$. For any $d_1d_2d_3\ldots \in D_b^\N$ let $\om^i=d_{i+1}d_{i+2}$ for all $i\in\N_0$. Then $\om=\om^0\om^1\ldots\in(D_b^m)^{\N_0}$ is a progressively overlapping sequence.
  Similarly, let $\tau^i=c_{i+1}d_{i+2}$ with $c_{i+1}\ne d_{i+1}$ for all $i\in\N_0$. Then $\tau=\tau^0\tau^1\ldots\in (D_b^m)^{\N_0}$ is a  totally distinct sequence.
So, there exist uncountably many progressively overlapping sequences and uncountably many totally distinct sequences.
\end{example}

When $\om\in(D_b^m)^{\N_0}$ is a progressively overlapping or a totally distinct sequence, the dimensions   in (\ref{eq:fractal-dimension}) are all equal for $K^\om$.
\begin{theorem}\label{thm:H=P=B=AOP}\mbox{}

\begin{enumerate}[{\rm(i)}]
  \item  If $\om \in (D_b^m)^{\mathbb{N}_0}$ is a progressively overlapping sequence, then the set $K^\om$ is $\frac{\log \la}{\log b}$-Ahlfors regular,
where $\la\in(b-1, b)$ is  Pisot number satisfying  the equation $x^m -(b-1)(x^{m-1}+x^{m-2} + \cdots +x+1)=0$.

  \item   If $\om \in (D_b^m)^{\mathbb{N}_0}$ is a totally distinct sequence, then the set $K^\om$ is $\frac{\log \eta}{\log b}$-Ahlfors regular,
  where $\eta\in(b-1, b)$ is the Pisot number satisfying the equation $x^m -bx^{m-1} +1=0$.
\end{enumerate}

\end{theorem}

\begin{remark}\label{rem:eventually-PO}
  If $\om=\om^0\om^1\ldots\in (D_b^m)^{\N_0}$ is eventually progressively overlapping, i.e.,
  there exists $k_0 \in \N$ such that $\om$ is progressively overlapping at position-$k$ for all $k \ge k_0$,
then the set $K^\om$ is also $\frac{\log\la}{\log b}$-Ahlfors regular. Similarly, if $\om$ is eventually totally distinct, then the set $K^\om$ is $\frac{\log\eta}{\log b}$-Ahlfors regular.
\end{remark}

Next, we show that the dimensions obtained in Theorem \ref{thm:H=P=B=AOP} (i) and (ii) are the upper and lower bounds for fractal dimensions of $K^\om$, respectively; and the Hausdorff and packing dimensions have the intermediate value property. Furthermore, the strict inequalities \[ \dim_{L} K^{\om}< \dim_{H} K^{\om} <\dim_{P} K^{\om}< \dim_{A} K^{\om}\] hold for infinitely  many sequences $\omega \in (D_b^m)^{\N_0}$.
In order to consider the Hausdorff dimension in the symbolic space $(D_b^m)^{\N_0}$, we equip $(D_b^m)^{\N_0}$ with the following metric
\begin{equation}\label{eq:metric-varrho}
  \varrho (\om,\tau)=b^{-m \cdot \inf\{ k \in \N_0: \omega^k \ne \tau^k \}} \quad\text{for}\quad\om=\om^0\om^1\ldots,~\tau=\tau^0\tau^1\ldots\in(D_b^m)^{\N_0}.
\end{equation}

\begin{theorem}\label{thm:intermediate value}\mbox{}

\begin{enumerate}[{\rm(i)}]
  \item  For any $\om\in(D_b^m)^{\N_0}$ we have
    \[ \frac{\log \eta}{\log b} \le \dim_L K^\om \le \dim_H K^{\om}\le\dim_P K^{\om}\le \dim_A K^{\om} \le \frac{\log \la}{\log b}, \]
    where $\la$ and $\eta$ are defined as in Theorem  \ref{thm:H=P=B=AOP}.
  \item For any $\frac{\log \eta}{\log b} \le \alpha \le \beta \le \frac{\log \la}{\log b}$ we have
    \[ \dim_H\Big\{\om\in (D_b^m)^{\N_0} : \dim_{H} K^{\om}=\alpha,~ \dim_{P} K^{\om}=\beta  \Big\}\ge \frac{1}{m}. \]
  \item We have
    \[ \dim_H\Big\{\om\in (D_b^m)^{\N_0} :\dim_{L} K^{\om}< \dim_{H} K^{\om}<\dim_{P} K^{\om}< \dim_{A} K^{\om} \Big\}\ge \frac{1}{m}.
    \]
\end{enumerate}

\end{theorem}
\begin{remark}\label{rem:multifractal}
      \begin{enumerate}[{\rm(i)}]
      \item Note that $\mathcal{A}=\set{A_{\mathbf d}: \mathbf d\in D_b^m}$ is irreducible in the following sense: there exists $N>0$ such that
    \begin{equation*}
      \sum_{k=1}^{N}\left(\sum_{\sd\in D_b^{m}}A_{\sd}\right)^k>\boldsymbol{0},
    \end{equation*}
    which means each element in the matrix $\sum_{k=1}^{N}\left(\sum_{\sd\in D_b^{m}}A_{\sd}\right)^k$ is positive. So by using Theorem 1.1 of   \cite{Feng-2009} we obtain the following multifractal result of $K^\om$. For any $\alpha\in[\log_b\eta, \log_b\la]$ set
    $
    L_{\alpha}:=\left\{\om\in \left(D_b^m\right)^{\N_0} :\dim_{H} K^{\om}=\dim_{P} K^{\om}=\alpha\right\}.
    $ Then \cite[Theorem 1.1]{Feng-2009} implies that
        \[
        \dim_H L_{\alpha}=\dim_P L_{\alpha}=\frac{1}{m}\inf_{q\in \mathbb{R}}\left\{-\alpha q+P_{\mathcal{A}}(q)\right\},
        \]
        where
       \[
          P_{\mathcal{A}}(q):=\lim_{n\rightarrow \infty}\frac{1}{n} \log_b \left(\sum_{\om^0\om^1\ldots\om^{n-1}\in \left(D_b^m\right)^n} \| A_{\om^0}A_{\om^1}\cdots A_{\om^{n-1}}\|^{q}\right).
          \]
    Adapting  the proof of \cite[Theorem 1.1]{Feng-2009} one can prove that for any $\al, \beta\in[\log_b\eta, \log_b\la]$ with $\al\le\beta$,
      \begin{align*}
        \dim_H \set{\om\in \left(D_b^m\right)^{\N_0} : \dim_{H} K^{\om}=\alpha,~ \dim_{P} K^{\om}=\beta  }&=\min \left\{\dim_HL_{\alpha},\dim_HL_{\beta} \right\},\\
         \dim_P \set{\om\in \left(D_b^m\right)^{\N_0} : \dim_{H} K^{\om}=\alpha,~ \dim_{P} K^{\om}=\beta  }&=\max \left\{\dim_HL_{\alpha}, \dim_HL_{\beta} \right\}.
      \end{align*}
      \item
     In the proof of Theorem \ref{thm:intermediate value} we obtain the following  estimation on the size of $\|A_{\om^0}\cdots A_{\om^{n-1}}\|$. For any $\om\in\left(D_b^m\right)^{\N_0}$ there exists a constant $C\ge 1$ such that for all $n\in \mathbb{N}$,
        \[
        C^{-1}\eta^{n}\le \|A_{\om^0}A_{\om^1}\cdots A_{\om^{n-1}}\|\le  C\la^{n}.
        \]
        \end{enumerate}
      \end{remark}

Finally, we show that if the progressively overlapping pattern and the totally distinct pattern occur periodically in the sequence $\om \in (D_b^m)^{\N_0}$, then the set $K^\om$ is Ahlfors regular. 
More precisely, for $p,q \in \N$  let $L_{p,q}$ be the set of sequences $\om \in (D_b^m)^{\N_0}$  that  is progressively overlapping at position-$k$ for $\ell(p+q)< k \le \ell(p+q)+p$ and is totally distinct at position-$k$ for $\ell(p+q) + p < k \le (\ell+1)(p+q)$, where $\ell \in \N_0$.
In other words, a sequence $\om=\om^0 \om^1 \om^2\ldots \in L_{p,q}$ with each $\om^k = \om^k_1 \om^k_2 \ldots \om^k_m \in D_b^m$ if  
it satisfies
\[
\begin{cases}
  \om^k_1 \ldots \om^k_{m-1} = \om^{k-1}_2 \ldots \om^{k-1}_m, & \mbox{for }\ell(p+q)< k \le \ell(p+q)+p,  \\
  \om^k_1 \ne \om^{k-1}_2,\; \om^k_2 \ldots \om^k_{m-1} = \om^{k-1}_3 \ldots \om^{k-1}_m  & \mbox{for }\ell(p+q) + p < k \le (\ell+1)(p+q) 
\end{cases}
\]
for all $\ell\in\N_0$.  Note that for any $(d_i)\in D_b^\N$ we can construct a sequence $\om=\om^0\om^1\om^2\ldots\in L_{p,q}$ by setting $\om^0_1\ldots \om^0_{m}=d_{1}\ldots d_{m}$, $\om^k_1\ldots \om^k_{m}=d_{k+1}\ldots d_{k+m}$ for $\ell(p+q)<k\le \ell(p+q)+p$, and $\om^k_1\ne d_{k+1}, \om^k_2\ldots \om^k_{m}=d_{k+2}\ldots d_{k+m}$ for $\ell(p+q)+p<k\le (\ell+1)(p+q)$, where $\ell\in\N_0$. This implies that for any  $p,q\in\N$ the set $L_{p,q}$ is uncountable.

\begin{theorem}\label{thm:L-p-q}
  For any $p,q \in \N$, there exists $\lambda_{p,q} >1$ such that the set $K^\om$ is $\frac{\log \lambda_{p,q}}{(p+q) \log b}$-Ahlfors regular for all $\om \in L_{p,q}$. 
  Moreover, the set \[ \bigg\{ \frac{\log \lambda_{p,q}}{(p+q)\log b}: p,q \in \N \bigg\} \] is dense in $[\frac{\log\eta}{\log b}, \frac{\log\lambda}{\log b}]$, where $\lambda$ and $\eta$ are defined as in Theorem \ref{thm:H=P=B=AOP}.
\end{theorem}

\subsection{Badly approximable numbers and non-recurrence points}
Based on the study of open dynamical systems with a moving hole, we consider the set of badly approximable numbers.

Given an integer $b\ge 3$, let $T_b: [0,1)\to[0,1);~ x\mapsto b x\pmod 1$ be the expanding map on the circle. For
 a sequence  of open balls $\{B_n\}_{n=0}^\f$ in $[0,1)$, the following set
\[
W(\set{B_n}):=\set{x\in[0,1): T_b^n (x)\in B_n\textrm{ for infinitely many }n}
\]
is well studied in Diophantine approximation,  which is   called the \emph{shrinking target problem} by Hill and Velani \cite{Hill-Velani-95}. Since $W(\set{B_n})=\limsup_{n\to\f}T_b^{-n}(B_n)$, the set $W(\set{B_n})$ is also called a \emph{limsup set}. It is well known that whether the set $W(\set{B_n})$ having zero or full Lebesgue measure depends only on the convergence or divergence of the series $\sum_{n=0}^{\f}r_n$ (cf.~\cite{Li-Liao-Velani-Zorin-23}), where $r_n$ is the radius of $B_n$ for all $n\in \mathbb{N}_0$. When $W(\set{B_n})$ has zero Lebesgue measure, or equivalently, when $\sum_{n=0}^\f r_n<+\f$,
Bugeaud and Wang \cite[Theorem 1.6]{Bugeaud-Wang-2014}  showed that
$ \dim_H W(\set{B_n})=\limsup_{n\to\f}\frac{n}{n-\log_b r_n}$. For more literature on the well approximable numbers we refer to \cite{Beresnevich-Ramirez-Velani-16, Li-Liao-Velani-Zorin-23, Wang-Wu-21} and the references therein.

When the series $\sum_{n=0}^{\f}r_n=\f$, the set $W(\set{B_n})$ has full Lebesgue measure, and then its complement
\begin{equation*}\label{def:Gamma}
\begin{split}
K\left(\set{B_n}\right)& := [0,1) \setminus W(\set{B_n})=  \set{x\in [0,1) : T_{b}^n (x)\notin B_n\textrm{ for all but finitely many }  n }
\end{split}
\end{equation*}
is a null set.  It is natural to consider the Hausdorff dimension of $K\left(\set{B_n}\right)$. The set $K(\set{B_n})$ is closely related to the set of badly approximable numbers (see \cite{Nielsen_1999}). Since $K(\set{B_n})=\liminf_{n\to\f} T_b^{-n}([0,1)\setminus B_n)$, we also call $K(\set{B_n})$ a \emph{liminf set} in Diophantine approximation.
Our next result characterizes when $K\left(\set{B_n}\right)$ has full Hausdorff dimension.

\begin{theorem}\label{thm:otwillappro}
\begin{enumerate}[{\rm(i)}]
  \item If $\lim_{n\to\f}\diam(B_n)=0$, then $\dim_H K(\set{B_n})=1$.
  \item If $\dim_H K(\set{B_n})=1$, then $\liminf_{n\to\f}\diam(B_n)=0$.
\end{enumerate}
\end{theorem}
\begin{remark}\label{rem:K(Bn)}
\begin{enumerate}[{\rm(i)}]
\item
If the limit $\lim_{n\to\f}\diam(B_n)$ exists, then by Theorem \ref{thm:otwillappro} it follows that
$
\dim_H K\left(\set{B_n}\right)=1$ if and only if $\lim_{n\to\infty}\diam(B_n)=0.
$

\item If the balls $B_n=B(y, \frac{1}{n+1})$ have the same center $y\in[0,1)$, then
\[K\left(\set{B_n}\right)\supseteq \set{x\in[0,1): y\notin\overline{\set{T_b^n(x): n\in\N_0}}}=:K\left(y\right),\]
 which has full Hausdorff dimension. This result was first proved by Urba\'nski \cite{Urbanski-1991}. In fact, Tseng \cite{Tseng-2009} proved a stronger result that the set    $K\left(y\right)$ is  $\alpha$-winning for some $\alpha\in(0,1/2]$ in the sense of Schmidt's game (cf.~\cite{Schmidt-1966}).
For a  general sequence of balls $\set{B_n}$ with $\lim_{n\to\infty}\diam(B_n)=0$, it might be interesting to ask whether  $K\left(\set{B_n}\right)$ is a winning set?
\end{enumerate}
 \end{remark}

The classical Poincar\'e's recurrence theorem states  that for Lebesgue almost every $x\in[0,1)$ its orbit $\set{T_b^n (x): n\in \mathbb{N}_0}$ returns to its neighborhood infinitely many times (cf.~\cite{Walters-82}). Boshernitzan \cite{Boshernitzan-93} described  this type of recurrence results   quantitatively.
Given a rate function $\phi: \N\to(0,\f)$,  Tan and Wang showed in \cite{Tan-Wang-2011} that the Hausdorff dimension of the  set
\[
R(\phi):=\set{x\in[0,1): |T_b^n (x)-x|<\phi(n)\textrm{ for infinitely many }n}
\]
is $\limsup_{n\to\f}\frac{n}{n-\log_b\phi(n)}$. In this paper we also consider its complement, that is
\[
E\left(\phi\right):=[0,1)\setminus R(\phi)=\set{x\in [0,1): |T_b^n (x)-x|\ge \phi(n)\textrm{ for all but finitely many }n}.
\]
\begin{theorem}
  \label{thm:notrecurrence}
 \begin{enumerate}[{\rm(i)}]
   \item If $\lim_{n\to\f}\phi(n)=0$, then $\dim_H E(\phi)=1$.
   \item If $\dim_H E(\phi)=1$, then $\liminf_{n\to\f}\phi(n)=0$.
 \end{enumerate}
\end{theorem}
\begin{remark}
  Similar to Remark \ref{rem:K(Bn)} (i),    if the limit $\lim_{n\to\f}\phi(n)$ exists, then   Theorem \ref{thm:notrecurrence} implies that
   $
  \dim_H E\left(\phi\right)=1$ if and only if $\lim_{n\to\infty}\phi(n)=0.
 $
\end{remark}

\subsection{Finiteness property for the joint spectral radius}
 Given a finite set of $d \times d$ real matrices $\mathcal{M} =\left\{M_1, \ldots, M_r\right\}$,    Rota and  Strang \cite{Rota-Strang-1960} introduced  the concept of  the \emph{joint spectral radius}
\begin{equation}\label{eq:joint spectral radius}
\widehat\rho(\mathcal {M}):=\lim_{n\to\infty} \sup_{i_1\ldots i_n\in\set{1,2,\ldots, r}^n}\left\|M_{i_1} \cdots M_{i_n}\right\|^{1 / n},
\end{equation}
which extends the notion of  spectral radius of a single matrix.
The set $\mathcal{M}$ is said to satisfy the \emph{finiteness property} if there exists a finite block $i_1 \ldots i_n\in\set{1,2,\ldots, r}^n$ such that \[\widehat\rho(\mathcal{M})=\rho\left(M_{i_1} \cdots M_{i_n}\right)^{1 / n},\]
where $\rho(M)$ denotes the spectral radius of a matrix $M$.
  Lagarias and   Wang \cite{Lagarias-Wang-1995} conjectured that  every finite set of $d \times d$ real matrices  possesses the finiteness property.
However, this conjecture was proven to be false in a general setting (cf.~\cite{Bochi-Morris-2015, Bousch-Mairesse-2002, Hare-Morris-Sidorov-2011}); but it still
 holds in many situations (cf.~\cite{Ahmadi-Jungers-2016, Cicone-Nicola-Zennaro-2010, Jungers-2009, Kozyakin-2016, Liu-Xiao-2013, Panti-Sclosa-2021}).
 In particular, if $\mathcal{M}$ is a finite set of $d \times d$ nonnegative matrices with $\widehat\rho(\mathcal{M})=\rho(\sum_{M\in \mathcal{M}}M)$, then $\mathcal{M}$ has the finiteness property (cf.~\cite{Michael-James-Philipp-Neil-2023}).

 Jungers and  Blondel \cite[Theorem 4]{Jungers-Blondel-2008} established the following key equivalence: the finiteness property holds for all sets of nonnegative rational square matrices if and only if it holds for all pairs of binary square matrices.
They also proved that the finiteness property holds for all pairs of $2\times 2$ binary matrices.
Note that each adjacency matrix $A_\sd$ of the directed graph $G_\sd$ is a $0-1$ matrix. Our  final result says that the set $\set{A_\sd: \sd\in D_b^m}$ of adjacency matrices satisfies the finiteness property.

\begin{theorem}\label{thm:finiteness property}
 The set $\mathcal{A}=\left\{A_{\sd}: \sd\in D_b^{m}\right\}$ has the finiteness property.
  Furthermore, for any periodic  and progressively overlapping sequence $\om=(\om^0\om^1\ldots\om^{n-1})^\f\in \left(D_b^m\right)^{\N_0}$ we have
    \[
      \widehat\rho(\mathcal A)=\rho(A_{\om^0}A_{\om^1}\cdots A_{\om^{n-1}})^{1/n}=\la,
      \]
    where $\la\in(b-1, b)$ is defined as in Theorem \ref{thm:H=P=B=AOP}.
\end{theorem}
Indeed Theorem \ref{thm:finiteness property} shows that the finiteness property $\widehat\rho(\mathcal A)=\rho(A_{\om^0}A_{\om^1}\ldots A_{\om^{n-1}})^{1/n}$ holds for infinitely many blocks $\om^0\om^1\ldots \om^{n-1}\in\left(D_b^m\right)^*, n\in\N$.

The rest of the paper is organized as follows. In the next section we consider the dimensions of $K^\om$ for a general $\om\in (D_b^m)^{\N_0}$, and prove Theorem \ref{general-result}. When $\om$ is progressively overlapping or totally distinct, we show in Section \ref{sec:dimension-equality-PO-TD} that all of these fractal dimensions of $K^\om$ are equal. In fact, we show in this case that $K^\om$ is Ahlfors regular (see Theorem \ref{thm:H=P=B=AOP}). In Section \ref{sec:intermediate-dimension}, we prove the intermediate value property for the dimensions of $K^\om$, and show that there are infinitely many $\om\in (D_b^m)^{\N_0}$ such that the dimensions of $K^\om$ are not equal (see Theorem \ref{thm:intermediate value}). 
In Section \ref{sec:regularity}, we study some special sequences $\om\in (D_b^m)^{\N_0}$ for which $\dim_H K^\om = \dim_P K^\om$, and prove Theorem \ref{thm:L-p-q}.
As applications, we prove in Section \ref{sec:applications} Theorems \ref{thm:otwillappro} and \ref{thm:notrecurrence} for badly approximable numbers in Diophantine approximation and non-recurrence points, respectively; and prove the finiteness property for the joint spectral radius of matrices (see Theorem \ref{thm:finiteness property}). Finally, in Section \ref{sec:final-remarks} we make some remarks on possible extensions of our work to higher dimensions or to homogeneous self-similar IFSs.

\section{Hausdorff and packing dimensions of $K^\om$}\label{sec:dimension-K}

In this section we will determine the Hausdorff and packing dimensions of $K^\om$ for a general sequence $\om\in(D_b^m)^{\N_0}$, and prove Theorem \ref{general-result}.
Note that $b \in \N_{\ge 3}$ and $m \in \N_{\ge 2}$.
For a sequence $\om=\om^0\om^1\ldots\in(D_b^m)^{\N_0}$, recall the definitions of $\Sigma^\om$ and $\Sigma_k^\om$ from (\ref{eq:survivor-set-symbolic}) and (\ref{eq:Sigma-k-om-def}), and one can easily check that
\begin{equation}\label{eq:Sigma-k-om}
  \Sigma_k^\om =
  \begin{cases}
    D_b^k, & \mbox{if } 1 \le k < m, \\
    \big\{ d_1d_2\ldots d_k\in D_b^k: d_{i+1}\ldots d_{i+m}\ne \om^i~\forall 0 \le i \le k-m \big\}, & \mbox{if } k \ge m.
  \end{cases}
\end{equation}
First we obtain the following lemma on the cardinality of $\Sigma_k^{\om}$. The sharp upper and lower bounds on $|\Sigma_{k}^{\om}|$ will be given in Lemma \ref{lem:bound-cardinality} later.

\begin{lemma}\label{lem:b-1<b}
For any $\om \in (D_b^m)^{\mathbb{N}_0}$ and for any $k\in\N$ we have
\begin{equation*}\label{eq:growth-Sigma-k}
 (b-1)|\Sigma_k^{\om}| \le |\Sigma_{k+1}^{\om}| \le b |\Sigma_k^{\om}|.
\end{equation*}
\end{lemma}

\begin{proof}
  Note by (\ref{eq:Sigma-k-om}) that $|\Sigma_k^{\om}|=b^k$ for all $1\le k<m$ and $|\Sigma_m^{\om}| = b^m -1$.
  Thus the inequalities hold for all $1\le k < m$.
  Now we assume $k\ge m$, and write $\om=\om^0\om^1\ldots\in (D_b^m)^{\N_0}$ with each $\om^i=\om^i_1\om^i_2\ldots \om^i_m \in D_b^m$.
  If $d_1d_2\ldots d_k\in\Sigma_k^{\om}$ and $d_{k+1}\in D_b\setminus\set{\om_m^{k-m+1}}$, then by (\ref{eq:Sigma-k-om}) we must have $d_1d_2\ldots d_k d_{k+1}\in\Sigma_{k+1}^{\om}$. So, we obtain $|\Sigma_{k+1}^{\om}|\ge (b-1)|\Sigma_k^{\om}|$. On the other hand, if $d_1 d_2 \ldots d_{k+1} \in \Sigma_{k+1}^{\om}$, then by (\ref{eq:Sigma-k-om}) its prefix $d_1 d_2 \ldots d_k \in \Sigma_k^{\om}$. Clearly, we have $d_{k+1} \in D_b$.
  It follows that $|\Sigma_{k+1}^{\om}|\le b |\Sigma_k^{\om}|$.
\end{proof}

For $\sc=c_1 c_2 \ldots c_k\in\Sigma_k^{\om}$ and $n\in\N$ let
\begin{equation}\label{eq:Gyn}
  \Sigma_{k+n}^\om(\sc) := \set{d_{1}d_{2}\ldots d_{k+n}\in \Sigma_{k+n}^{\om}:   d_1 d_2 \ldots d_k=\sc}.
\end{equation}
The average cardinality of $\Sigma_{k+n}^\om(\sc)$ for $\sc\in\Sigma_k^\om$ is $|\Sigma_{k+n}^\om| / |\Sigma_k^\om|$.
Next, we show that each set $\Sigma_{k+n}^\om(\sc)$ has the cardinality comparable with the average value.

\begin{lemma}\label{lem:key-estimate}
  Let $\om \in (D_b^m)^{\mathbb{N}_0}$ and $k \in \N_{\geq m}$. Then for $\sc\in\Sigma_k^{\om}$ and $n \in \N$ we have
  \begin{equation}\label{eq:estimate}
    b^{-(m-1)} \frac{|\Sigma_{k+n}^\om|}{|\Sigma_k^\om|} \le |\Sigma_{k+n}^\om(\sc)| \le b^{m-1} \frac{|\Sigma_{k+n}^{\om}|}{|\Sigma_k^{\om}|}.
  \end{equation}
\end{lemma}
\begin{proof}
By Lemma \ref{lem:b-1<b} we have \[ (b-1)^n \le \frac{ |\Sigma_{k+n}^\om|}{|\Sigma_k^\om|} \le b^n\quad \forall n\in\N. \]
Thus, for $1\le n < m$ we have \[ b^{-(m-1)} \frac{|\Sigma_{k+n}^\om|}{|\Sigma_k^\om|} \le 1 \le |\Sigma_{k+n}^\om(\sc)| \le b^{m-1} \le b^{m-1} \frac{|\Sigma_{k+n}^{\om}|}{|\Sigma_k^{\om}|}. \]
In the following we assume that $n \ge m$.

Write $\om=\om^0\om^1\ldots\in (D_b^m)^{\N_0}$ with each $\om^i=\om^i_1\om^i_2\ldots \om^i_m \in D_b^m$.
Since $b\ge 3$, we can choose a word $u_1\ldots u_{m-1}\in D_b^{m-1}$ such that
\begin{equation}\label{eq:ui}
u_i\notin\set{\om^{k-m+i}_m, \om^{k-1+i}_1}\quad\forall 1\le i\le m-1.
\end{equation}
Let $\tau:=\om^{k+m-1}\om^{k+m}\ldots\in (D_b^m)^{\N_0}$.
We claim that
\[
\mathbf{c} u_1\ldots u_{m-1} \mathbf{d}\in \Sigma_{k+n}^{\om} \quad\text{for any}~ \sd \in \Sigma_{n-m+1}^\tau.
\]

Write $v_1 v_2 \ldots v_{k+n}=\mathbf{c} u_1\ldots u_{m-1} \mathbf{d}$ for some $\sd \in \Sigma_{n-m+1}^\tau$. By (\ref{eq:Sigma-k-om}) it suffices to show that
\begin{equation}\label{eq:condition-1}
  v_{i+1} \ldots v_{i+m} \ne \om^i \quad \forall 0 \le i \le k+n-m.
\end{equation}
For $0 \le i \le k-m$, (\ref{eq:condition-1}) follows from $\sc\in\Sigma_k^{\om}$;
for $k-m +1 \le i \le k-1$, (\ref{eq:condition-1}) follows from (\ref{eq:ui}) that $v_{i+m}=u_{i+m-k}\ne \om_m^i$; for $k\le i\le \min\{k+n-m,k+m -2\}$, (\ref{eq:condition-1}) follows from (\ref{eq:ui}) that $v_{i+1}=u_{i+1-k}\ne \om_1^i$; and for $k+m-1 \le i \le k+n -m$ if possible, the fact that $\sd \in \Sigma_{n-m+1}^\tau$ implies (\ref{eq:condition-1}).
The claim has been established.
Thus, we obtain that
\begin{equation}\label{eq:nov28-1}
  |\Sigma_{k+n}^\om(\sc)|\ge |\Sigma_{n-m+1}^\tau|.
\end{equation}
On the other hand, for any word $d_1d_2\ldots d_{k+n}\in\Sigma_{k+n}^\om(\sc)$ we have
\[
d_1\ldots d_k=\sc,\quad d_{k+1}\ldots d_{k+m-1}\in D_b^{m-1},\quad\textrm{and}\quad d_{k+m}\ldots d_{k+n}\in\Sigma_{n-m+1}^\tau,
\]
which implies that
\begin{equation}\label{eq:nov28-2}
|\Sigma_{k+n}^\om( \sc)|\le b^{m-1}|\Sigma_{n-m+1}^\tau|.
\end{equation}

Observe that
\[
\Sigma_{k+n}^{\om} =\bigcup_{\widetilde{\sc}\in\Sigma_k^{\om}}\Sigma_{k+n}^\om(\widetilde{\sc}),
\]
where the unions are pairwise disjoint.
Then by (\ref{eq:nov28-1}) and (\ref{eq:nov28-2}) we obtain that
\[
|\Sigma_k^{\om}|\cdot |\Sigma_{n-m+1}^\tau|\le |\Sigma_{k+n}^{\om}|\le b^{m-1}|\Sigma_k^{\om}|\cdot |\Sigma_{n-m+1}^\tau|,
\]
which implies that
\[
b^{-(m-1)}\frac{|\Sigma_{k+n}^{\om}|}{|\Sigma_k^{\om}|}\le |\Sigma_{n-m+1}^\tau|\le \frac{|\Sigma_{k+n}^{\om}|}{|\Sigma_k^{\om}|}.
\]
This together with (\ref{eq:nov28-1}) and (\ref{eq:nov28-2}) implies (\ref{eq:estimate}).
\end{proof}

For $\sd \in D_b^m$, recall that $A_{\sd}$ is the adjacency matrix of the directed graph $G_{\sd}$, which describes the shift of finite type $X_\sd$. Recall that for a matrix $A=(a_{i,j})$ we use the norm $\| A \|=\sum_{i,j}|a_{i,j}|$.

\begin{lemma}\label{lem:counting}
  For any $\om=\om^0\om^1\ldots\in (D_b^m)^{\N_0}$ and $k \in \N_{\ge m}$ we have
  \[ |\Sigma_k^\om|=\|A_{\om^0}A_{\om^1}\cdots A_{\om^{k-m}}\|. \]
\end{lemma}
\begin{proof}
  For $i \in \N_0$, write \[ A_{\om^i} = \big( a_i(\mathbf{I},\mathbf{J}) \big)_{\mathbf{I},\mathbf{J}\in D_b^{m-1}}. \]
  Then we have \[ \|A_{\om^0}A_{\om^1}\cdots A_{\om^{k-m}}\| = \sum_{\mathbf{I}_0,\mathbf{I}_1,\ldots,\mathbf{I}_{k-m+1} \in D_b^{m-1}} \prod_{i=0}^{k-m} a_i(\mathbf{I}_i, \mathbf{I}_{i+1}). \]
  Every block $d_1 d_2 \ldots d_k \in D_b^k$ corresponds to a unique term in the above sum
  \[ \prod_{i=0}^{k-m} a_i\big( d_{i+1} \ldots d_{i+m-1},~ d_{i+2} \ldots d_{i+m} \big). \]
  By (\ref{eq:Sigma-k-om}), we have $d_1 d_2 \ldots d_k \in \Sigma_k^\om$ if and only if $d_{i+1} \ldots d_{i+m-1} d_{i+m} \ne \om^{i}$ for all $0\le i \le k-m$.
  By the definition of $A_{\om^i}$, it is also equivalent to \[ \prod_{i=0}^{k-m} a_i\big( d_{i+1} \ldots d_{i+m-1}, d_{i+2} \ldots d_{i+m} \big)=1. \]
  The desired equality follows directly.
\end{proof}

\begin{proof}[Proof of Theorem \ref{general-result}]
Note by (\ref{eq:survivor-set-K}) that  for any $k\in \mathbb{N}$, $K^\om\subset\bigcup_{\sd\in \Sigma_k^\om}I_\sd$ and $K^\om \cap I_\sd \ne \emptyset$ for all $\sd \in \Sigma_k^\om$.
By the definition of box dimension (cf.~\cite{Falconer-14}) and Lemma \ref{lem:counting}, we obtain
$$\underline{\dim}_B K^{\om} = \liminf_{n \to \f} \frac{\log |\Sigma_n^{\om}|}{n \log b }=\liminf_{n\to\infty}\frac{1}{n \log b}\log \|A_{\om^0}A_{\om^1}\cdots A_{\om^{n-1}}\|,$$
and
$$\overline{\dim}_B K^{\om} = \limsup_{n \to \f} \frac{\log |\Sigma_n^{\om}|}{n \log b}=\limsup_{n\to\infty}\frac{1}{n \log b}\log\|A_{\om^0}A_{\om^1}\cdots A_{\om^{n-1}}\|.$$

Note that $\dim_H K^\om \le \underline{\dim}_B K^{\om} =: s$. We only need to show that $\dim_{H} K^\om \ge s$.
By Lemma \ref{lem:b-1<b} and using $b\in\N_{\ge 3}$, we have $s\ge \log(b-1)/\log b >0$.
Take $0 < t < s$. Then we can find $k_0 \in \N_{\geq m}$ such that
\begin{equation}\label{eq:t-n}
  |\Sigma_k^{\om}| \ge b^{k t}\quad \forall k\ge k_0.
\end{equation}
We will prove that $\mathcal{H}^t\big( K^\om \big) >0$.
Note that $K^\om$ is a compact set. By the definition of Hausdorff measure (cf.~\cite{Falconer-14}) it suffices to show that
\begin{equation}\label{eq:H-t}
\sum_{i=1}^q\big(\diam(U_i)\big)^{t} \geq b^{-(m+s)}
\end{equation}
for any finite open $b^{-k_0}$-covering $\left\{U_i\right\}_{i=1}^q$ of $K^\om$.

Let $\left\{U_i\right\}_{i=1}^q$ be an open $b^{-k_0}$-covering of $K^\om$.
For each $1\le i \le q$, choose $k_i\in \N_{\ge k_0}$ such that
\begin{equation}\label{eq:U-i}
{b}^{-(k_i+1)} \le \diam(U_i)< {b}^{-k_i}.
\end{equation}
Fix an integer $k > \max\{ k_1,k_2,\ldots, k_q \}$.
Since $\left\{U_i\right\}_{i=1}^q$ is a covering of $K^\om$, by (\ref{eq:survivor-set-K}) we have
\begin{equation}\label{eq:covering-K}
  \Sigma_k^\om = \bigcup_{i=1}^q \{ \sd \in \Sigma_k^\om:   I_\sd \cap U_i \ne \emptyset\}.
\end{equation}
For each $1 \le i \le q$, since $\diam(U_i) < b^{-k_i}$, the set $U_i$ intersects at most two intervals in $\{ I_\sd : \sd \in \Sigma_{k_i}^\om \}$, denoted by $I_{\sc_i}, I_{\widetilde{\sc}_i}$ with $\sc_i, \widetilde{\sc}_i \in \Sigma_{k_i}^\om$.
Recall the definitions of $\Sigma_{k}^\om(\sc_i)$ and $ \Sigma_{k}^\om(\widetilde{\sc}_i)$ in (\ref{eq:Gyn}).
Then for each $1\le i\le q$ we have
 \[
  \{ \sd \in \Sigma_k^\om:  I_\sd \cap U_i \ne \emptyset\} \subset \Sigma_{k}^\om(\sc_i) \cup \Sigma_{k}^\om(\widetilde{\sc}_i).
   \]
 By (\ref{eq:covering-K}) and Lemma \ref{lem:key-estimate}  we obtain
\[
|\Sigma_k^\om|\le\sum_{i=1}^{q}|\{ \sd \in \Sigma_k^\om:  I_\sd \cap U_i \ne \emptyset\}| \le \sum_{i=1}^{q}\left(|\Sigma_{k}^\om(\sc_i)| +|\Sigma_{k}^\om(\widetilde{\sc}_i)|\right)\le 2 b^{m-1} |\Sigma_k^\om|\sum_{i=1}^{q}|\Sigma_{k_i}^\om|^{-1}.
 \]
Note that $b\geq 3$. It follows that
 \begin{equation}
   \label{eq:march3-1}
   \sum_{i=1}^{q}|\Sigma_{k_i}^\om|^{-1}\ge b^{-m}.
 \end{equation}
Since $k_i \ge k_0$, by (\ref{eq:t-n}) and (\ref{eq:U-i}) we have
$|\Sigma_{k_i}^\om|^{-1}\le b^{-k_i t}<b^s \big(\diam(U_i)\big)^t$.
Then (\ref{eq:H-t}) follows directly from (\ref{eq:march3-1}).
Thus we obtain $\dim_H K^\om \ge t$.
Since $t\in(0,s)$ was chosen arbitrarily, we conclude that $\dim_{H} K^\om \ge s$.

In order to show $\dim_P K^\om =\overline{\dim}_B K^\om$, by \cite[Corollary 3.10]{Falconer-14} it suffices to show that for any open set $V$ with $V \cap K^\om \ne \emptyset$,
\begin{equation}\label{eq:packing-dimension}
  \overline{\dim}_B \big( V \cap K^\om \big) = \overline{\dim}_B K^\om = \limsup_{k \to \f} \frac{\log |\Sigma_k^\om|}{k \log b}.
\end{equation}
Take an open set $V$ with $V \cap K^\om \ne \emptyset$.
Then we can find $k_0 \in \N_{\geq m}$ and $\sc \in \Sigma_{k_0}^\om$ such that $ I_\sc \cap K^\om \subset V \cap K^\om.$
Recall the definition of $\Sigma_{k}^\om(\sc)$ in (\ref{eq:Gyn}).
Then for any $k >k_0$ we have
 $$\Sigma_k^\om(\sc) \subset \big\{\sd \in D_b^k:  I_\sd \cap (I_\sc \cap K^\om) \ne \emptyset \big\}.$$
Thus, we obtain
$$\overline{\dim}_B \big( V \cap K^\om \big) \ge \overline{\dim}_B \big( I_\sc \cap K^\om \big) \ge \limsup_{k \to \f} \frac{\log | \Sigma_{k}^\om(\sc)|}{k\log b}.$$
It follows from Lemma \ref{lem:key-estimate} that
\[ \overline{\dim}_B \big( V \cap K^\om \big) \ge \limsup_{k \to \f} \frac{\log |\Sigma_k^\om| -\log |\Sigma_{k_0}^\om| - (m-1)\log b}{k \log b} = \limsup_{k \to \f} \frac{\log |\Sigma_k^\om|}{k \log b}. \]
We clearly have \[ \overline{\dim}_B \big( V \cap K^\om \big) \le \overline{\dim}_B K^\om = \limsup_{k \to \f} \frac{\log |\Sigma_k^\om|}{k \log b}.\]
Therefore, we obtain (\ref{eq:packing-dimension}), completing the proof.
\end{proof}
In general, we can not expect $\dim_H K^\om=\dim_P K^\om$. However, if $\om$ is a progressively overlapping sequence or a totally distinct sequence, then we have the equality $\dim_H K^\om=\dim_P K^\om$, which will be discussed  in the next section.

\section{Dimensions of $K^\om$  for progressively overlapping or totally distinct $\om$}\label{sec:dimension-equality-PO-TD}
In this section we will prove the dimension equality of $K^\om$ for progressively overlapping and totally distinct sequences $\om\in(D_b^m)^{\N_0}$, and establish Theorem \ref{thm:H=P=B=AOP}.
Recall the definitions of progressively overlapping and totally distinct sequences in subsection \ref{subsec:open-dynamics}, and keep in mind that $b \in \N_{\ge 3}$ and $m \in \N_{\ge 2}$.

\begin{proposition}\label{prop:Ahlfors-regular}
  For $\omega \in (D_b^m)^{\N_0}$, if there exist $\beta > 1$ and $C>1$ such that
  \[ C^{-1}\beta^k \le |\Sigma_k^\om| \le C \beta^k \quad \forall k \in \N, \]
  then the set $K^\om$ is $\frac{\log \beta}{\log b}$-Ahlfors regular.
\end{proposition}
\begin{proof}
  Write $s= \log \beta/\log b$. It suffices to show that there exists a constant $c\ge 1$ such that for all $0< r< b^{-m}$ and for all $x \in K^\om$,
  \begin{equation}\label{eq:Ahlfors-regular}
    c^{-1} r^s \le \mathcal{H}^s\big( K^\om \cap B(x,r) \big) \le c r^s.
  \end{equation}

  Fix $0< r < b^{-m}$ and $x \in K^\om$. Choose $k_0 \in \N_{\ge m}$ such that
  \begin{equation}\label{eq:r}
    \frac{1}{b^{k_0+1}} \le r < \frac{1}{b^{k_0}}.
  \end{equation}
  We first show the upper bound in (\ref{eq:Ahlfors-regular}).
  Since $r< b^{-k_0}$, the ball $B(x,r)$ intersects at most three intervals in $\{I_\sd : \sd\in \Sigma_{k_0}^\om\}$, denoted by $I_{\sc_1},I_{\sc_2},I_{\sc_3}$ with $\sc_i \in \Sigma_{k_0}^\om$ for $1\le i\le 3$.
  That is, $K^\om \cap B(x,r) \subset I_{\sc_1} \cup I_{\sc_2} \cup I_{\sc_3}$. Recall the definition of $\Sigma_{k_0+n}^\om(\sc_i)$ in (\ref{eq:Gyn}).
  For $n \in \N$, the family of intervals $\mathcal{I}_n:=\{I_\sd : \sd \in \Sigma_{k_0+n}^\om(\sc_1)\cup \Sigma_{k_0+n}^\om(\sc_2)\cup\Sigma_{k_0+n}^\om(\sc_3)\}$ forms a $b^{-(k_0+n)}$-covering of $K^\om \cap B(x,r)$.
  By the assumption and using Lemma \ref{lem:key-estimate} and (\ref{eq:r}), we have
  \begin{align*}
    \sum_{I \in \mathcal{I}_n} \big( \mathrm{diam}(I)\big)^s & \le |\mathcal{I}_n| b^{-s(k_0+n)} \le \big( |\Sigma_{k_0+n}^\om(\sc_1)| +  |\Sigma_{k_0+n}^\om(\sc_2)|+|\Sigma_{k_0+n}^\om(\sc_3)| \big) b^{-s(k_0+n)} \\
    & \le 3b^{m-1} \frac{|\Sigma_{k_0+n}^\om|}{|\Sigma_{k_0}^\om|} b^{-s(k_0+n)} \le 3 b^{m-1}C^2 \cdot \beta^n b^{-s(k_0+n)} \\
    & \le 3 b^{m+s-1} C^2 \cdot r^s.
  \end{align*}
  This implies that \[ \mathcal{H}^s\big( K^\om \cap B(x,r) \big) \le 3 \beta b^{m-1} C^2 \cdot r^s.\]

  In order to show the lower bound in (\ref{eq:Ahlfors-regular}), noting that $K^\om \cap B(x,r)$ is compact, it suffices to show that for any finite open $b^{-k_0}$-covering $\left\{U_i\right\}_{i=1}^q$ of $K^\om \cap B(x,r)$,
  \begin{equation*}\label{eq:H-s}
  \sum_{i=1}^q\big( \diam(U_i) \big)^{s} \geq c^{-1}  r^s.
  \end{equation*}

  Let $\left\{U_i\right\}_{i=1}^q$ be an open $b^{-k_0}$-covering of $K^\om$.
  For each $1\le i \le q$, choose $k_i\in \N_{\ge k_0}$ such that
  \begin{equation}\label{eq:diam-U-i}
  {b}^{-(k_i+1)} \le \diam(U_i)< {b}^{-k_i}.
  \end{equation}
  Then the set $U_i$ intersects at most two intervals in $\{ I_\sd : \sd \in \Sigma_{k_i}^\om \}$, denoted by $I_{\sc_i}, I_{\widetilde{\sc}_i}$ with $\sc_i, \widetilde{\sc}_i \in \Sigma_{k_i}^\om$.
  Since $r\ge b^{-(k_0+1)}$, we can find $\sc \in \Sigma_{k_0+1}^\om$ such that $K^\om \cap I_{\sc} \subset K^\om \cap B(x,r)$.
  Fix an integer $k> \max\{k_0+1,k_1,\ldots,k_q\}$.
  Recall the definitions of $\Sigma_{k}^\om(\sc_i)$, $ \Sigma_{k}^\om(\widetilde{\sc}_i)$, and  $\Sigma_{k}^\om(\sc)$ in (\ref{eq:Gyn}).
  Note that $K^\om \cap I_{\sc} \subset K^\om \cap B(x,r)$ and $\{U_i\}_{i=1}^q$ is a cover of $K^\om $.
  Then we have \[ \Sigma_{k}^\om(\sc) \subset \{\sd\in \Sigma_k^\om : I_\sd \cap K^\om \cap B(x,r)\ne\emptyset \} \subset \bigcup_{i=1}^q \big( \Sigma_{k}^\om(\sc_i) \cup \Sigma_{k}^\om(\widetilde{\sc}_i) \big). \]
  It follows that \[ |\Sigma_{k}^\om(\sc)| \le \sum_{i=1}^q \big( |\Sigma_{k}^\om(\sc_i)| + |\Sigma_{k}^\om(\widetilde{\sc}_i)| \big). \]
  By Lemma \ref{lem:key-estimate}, we obtain \[ b^{-(m-1)} \frac{|\Sigma_k^\om|}{|\Sigma_{k_0+1}^\om|}\le|\Sigma_k^\om(\mathbf c)| \le 2 b^{m-1} \sum_{i=1}^q \frac{|\Sigma_k^\om|}{|\Sigma_{k_i}^\om|}. \]
  By the assumption this implies \[ \frac{\beta^{-(k_0+1)}}{b^{m-1} C } \le 2b^{m-1} C \sum_{i=1}^{q} \beta^{-k_i}.  \]
  This together with (\ref{eq:r}) and (\ref{eq:diam-U-i}) implies that
  \[ \sum_{i=1}^q\big( \diam(U_i) \big)^{s} \ge \sum_{i=1}^q b^{-s(k_i+1)} = \sum_{i=1}^q \beta^{-(k_i+1)}
    \ge \frac{\beta^{-(k_0+2)}}{2b^{2m-2} C^2 } \ge \frac{r^s}{2 \beta^2 b^{2m-2} C^2 }. \]
  Hence, (\ref{eq:Ahlfors-regular}) holds by taking $c=\max\set{3\beta b^{m-1}C^2, 2\beta^2 b^{2m-2} C^2}$, completing the proof.
\end{proof}

\begin{lemma}\label{lem:bound-cardinality}
 For any $\om \in(D_b^m)^{\N_0}$ and any $k\in\N$ we have
 \[ b|\Sigma_{k+m-1}^\om|-|\Sigma_{k}^\om| \le |\Sigma_{k+m}^\om|\le (b-1)\sum_{j=0}^{m-1} |\Sigma_{k+j}^\om|. \]
 Moreover, the lower bound equality holds if $\om$ is totally distinct at position-$k$, and the upper bound equality holds if $\om$ is progressively overlapping at position-$k$.
\end{lemma}
\begin{proof}
  Write $\om = \omega^0 \omega^1 \omega^2 \ldots \in(D_b^m)^{\N_0}$ with each $\om^i=\om_1^i \om_2^i \ldots \om_m^i \in D_b^m$, and fix $k\in\N$.

  (i) We first show the lower bound.
  Let \[
H_{k+m}^\om:=\big\{ d_1d_2\ldots d_{k+m}\in D_b^{k+m}: d_1d_2\ldots d_k\in\Sigma_k^\om, d_{k+1}d_{k+2}\ldots d_{k+m}=\om^k \big\}.
\]
Then it's clear that $H_{k+m}^\om\cap \Sigma_{k+m}^\om=\emptyset$.
For any $d_1\ldots d_{k+m-1}\in \Sigma_{k+m-1}^\om$ and any $d_{k+m}\in D_b$, if $d_{k+1}d_{k+2}\ldots d_{k+m}\neq \om_1^{k}\om_2^{k}\ldots\om_m^k$, then $d_1\ldots d_{k+m}\in\Sigma_{k+m}^\om$; otherwise, we have $d_1\ldots d_{k+m}\in H_{k+m}^\om$ because $d_1\ldots d_{k}\in\Sigma_{k}^\om$.
It follows that
\begin{equation}\label{eq:lower-1}
  \Sigma_{k+m-1}^\om\times D_b \subset H_{k+m}^\om\cup\Sigma_{k+m}^\om,
\end{equation}
which implies that
\begin{equation*}
|\Sigma_{k+m}^\om|\ge|\Sigma_{k+m-1}^\om\times D_b|-|H_{k+m}^\om|= b|\Sigma_{k+m-1}^\om|-|\Sigma_{k}^\om|.
\end{equation*}

Next, we assume that $\om$ is totally distinct at position-$k$. In order to show the equality, by (\ref{eq:lower-1}) it suffices to show
\[ H_{k+m}^\om\cup\Sigma_{k+m}^\om \subset \Sigma_{k+m-1}^\om\times D_b. \]
We clearly have $\Sigma_{k+m}^\om \subset \Sigma_{k+m-1}^\om\times D_b$. It remains to show
\begin{equation}\label{eq:lower-2}
  H_{k+m}^\om\subset \Sigma_{k+m-1}^\om\times D_b.
\end{equation}
Take a word $d_1d_2\ldots d_{k+m}\in H_{k+m}^\om$. Then $d_1\ldots d_k\in\Sigma_k^\om$ and $d_{k+1}d_{k+2}\ldots d_{k+m}=\om_1^{k}\om_2^{k}\ldots\om_m^k$.
Since $\om$ is totally distinct at position-$k$, we have \[ \omega^k_1 \ldots \omega^k_{m-j} \ne \omega^{k-j}_{j+1} \ldots \omega^{k-j}_m\quad \forall 1 \le j \le \min\{ k, m-1\}. \]
It follows that for $\max\{k-m+1,0\} \le i < k$, \[ d_{i+1} \ldots d_{i+m} = d_{i+1} \ldots d_k \omega^k_{1} \ldots \omega^k_{m-k+i} \ne \omega^i.  \]
This implies that $d_1\ldots d_{k+m-1}\in\Sigma_{k+m-1}^\om$. Thus, we obtain (\ref{eq:lower-2}) as desired.

  (ii) Now we prove the upper bound.
  Note by (\ref{eq:Sigma-k-om}) that each $d_1d_2\ldots d_{k+m}\in \Sigma_{k+m}^\om$ satisfies $d_{k+1} d_{k+2}\ldots d_{k+m}\ne \om_1^k\om_2^k\ldots \om_m^k$. Conditioned on the maximum index $j$ such that $d_{k+j}\ne \om_j^k$ we make the following partition of $\Sigma_{k+m}^\om$ as
  \begin{equation}\label{eq:upper-1}
    \Sigma_{k+m}^\om=\bigcup_{j=1}^m \Sigma_{k+m}^\om(j),
  \end{equation}
  where for $1\le j \le m$, \[ \Sigma_{k+m}^\om(j):=\big\{d_1d_2\ldots d_{k+m}\in \Sigma_{k+m}^\om: d_{k+j}\ne \om^k_j, d_{k+j+1}\ldots d_{k+m}=\om^k_{j+1}\ldots \om^k_{m}\big\}.\]

  If $d_1\ldots d_{k+m}\in \Sigma_{k+m}^\om(j)$, then it has the form $d_1\ldots d_{k+j-1}d_{k+j}\om^k_{j+1}\om_{j+2}^k\ldots \om_m^k$ with $d_1\ldots d_{k+j-1}\in\Sigma_{k+j-1}^\om$ and $d_{k+j}\in D_b\setminus\big\{\om^k_j\big\}$. So,
 $|\Sigma_{k+m}^\om(j)|\le (b-1)|\Sigma_{k+j-1}^\om|$ for all $1\le j\le m$.
  By (\ref{eq:upper-1}) it follows that
  \begin{equation}\label{eq:upper-2}
   |\Sigma_{k+m}^\om|=\sum_{j=1}^m|\Sigma_{k+m}^\om(j)|\le (b-1)\sum_{j=0}^{m-1}|\Sigma_{k+j}^\om|.
  \end{equation}

  Next, we assume that $\om$ is progressively overlapping at position-$k$.
  Given $1\le j\le m$, we take a word $d_1d_2\ldots d_{k+j-1}\in\Sigma_{k+j-1}^\om$ and take a digit $d_{k+j}\in D_b\setminus\big\{\om^k_j\big\}$. We claim that
  \[
  c_1c_2\ldots c_{k+m}:=d_1d_2\ldots d_{k+j-1}d_{k+j} \om^k_{j+1}\om^k_{j+2}\ldots \om_m^k\in \Sigma_{k+m}^\om(j).
  \]
  By the definition of $\Sigma_{k+m}^\om(j)$ it suffices to prove that the word $c_1\ldots c_{k+m}\in \Sigma_{k+m}^\om$.
  Note that $c_{1}\ldots c_{k+j-1}=d_{1}\ldots d_{k+j-1}\in \Sigma_{k+j-1}^\om$.
  We only need to verify \begin{equation}\label{eq:april-26-1} c_{i+1}c_{i+2}\ldots c_{i+m}\ne \om_1^i\om_2^i\ldots \om_m^i\quad \forall  \max \{k+j-m,0\} \le i\le k. \end{equation}
  Since $\om$ is progressively overlapping at position-$k$, we have \[ \omega^k_1 \ldots \omega^k_{m-\ell} = \omega^{k-\ell}_{\ell+1} \ldots \omega^{k-\ell}_m\quad \forall 1 \le \ell \le \min\{ k, m-1\}. \]
  Thus, for $\max \left\{k+j-m,0\right\}\le i\le k$ we have
  \[ c_{i+(k+j-i)}=c_{k+j}=d_{k+j}\ne\om^k_j=\om_{k+j-i}^i.\]
  This implies that $c_{i+1}c_{i+2}\ldots c_{i+m}\ne \om_1^i\om_2^i\ldots \om_m^i$, which establishes (\ref{eq:april-26-1}).

   By (\ref{eq:april-26-1}) it follows that
   \[
   \set{d_1\ldots d_{k+j}\om_{j+1}^k\ldots \om_m^k: d_1\ldots d_{k+j-1}\in\Sigma_{k+j-1}^\om, d_{k+j}\in D_b\setminus\{ \om_j^k\}} \subset\Sigma_{k+m}^\om(j),
   \]
   which implies that $(b-1)|\Sigma_{k+j-1}^\om|\le |\Sigma_{k+m}^\om(j)|$. By (\ref{eq:upper-1}) we obtain that
   \[
  |\Sigma_{k+m}^\om|= \sum_{j=1}^{m}|\Sigma_{k+m}^\om(j)|\ge (b-1)\sum_{j=0}^{m-1}|\Sigma_{k+j}^\om|.
   \]
  This together with (\ref{eq:upper-2}) proves the equality.
\end{proof}

Recall that a \emph{Pisot number} is an algebraic integer greater than $1$ whose algebraic conjugates are of modulus strictly less than $1$.
\begin{lemma}\label{lem:pisot-number}
Given $b\in\N_{\ge 3}$ and $m \in \N_{\ge 2}$, let $\la$ and $\eta$   be the largest real root of
\[x^m=(b-1)(x^{m-1}+x^{m-2}+\cdots+x+1)\quad\textrm{and}\quad x^m=bx^{m-1}-1,\]
respectively. Then $b-1<\eta<\la<b$, and both $\la$ and $\eta$ are Pisot numbers.
\end{lemma}
\begin{proof}
Let $f(x) := x^m-(b-1)(x^{m-1}+x^{m-2}+\cdots+x+1)$.
It is easy to verify that $f(b-1)<0$ and $f(x)\geq f(b)=1$ for all $x\geq b$.
By the intermediate value theorem the largest real root $\la \in(b-1, b)$.
Write $\widetilde{f}(x) := (x-1)f(x) = x^{m+1} - b x^m +b-1 = f_1(x) - f_2(x)$ with $f_1(x) := x^{m+1}$ and $f_2(x) := b x^m -(b-1)$.
Note that \[ \lim_{x \to 1^+} \frac{f_2(x)-1}{f_1(x)-1} = \frac{mb}{m+1} >1.\]
There is $r_0 > 1$ such that $b r^m - (b-1)=f_2(r)>f_1(r)= r^{m+1}$ for all $1 < r < r_0$.
Fix $1 < r < r_0$. For any $z \in \mathbb{C}$ with $|z| =r$, we have $$|f_2(z)|= |bz^{m}-(b-1)|\geq b r^m-(b-1) > r^{m+1}=|f_1(z)|.$$
Since $f_2(x)=b x^m-(b-1)$ has $m$ zeros in $\{ z \in \mathbb{C}:|z| < r\}$, by Rouche's theorem (cf.~\cite{Stein-Shakarchi-2003}) we obtain that $\widetilde{f}(x)$ also has $m$ zeros in $\{ z \in \mathbb{C}:|z| < r\}$.
By the arbitrariness of $r>1$, we conclude that $\widetilde{f}(x)$ has $m$ zeros in $\{ z \in \mathbb{C}:|z|\le 1\}$.
If $z \in \mathbb{C}$ with $|z| =1$ and $\widetilde{f}(z)=0$, then we have
\[ 0 = |\widetilde{f}(z)| = |z^{m+1} - b z^m +b-1| \ge |b -z| - (b-1) \ge b -|z| -(b-1) =0, \]
which implies that $z=1$.
Note by using $m\in\N_{\ge 2}, b\in\N_{\ge 3}$ that the derivative $\widetilde{f}'(1) \ne 0$. Thus, $z = 1$ is the unique zero of $\widetilde{f}$ on the unit circle. This implies that $f(x) = \widetilde{f}(x)/(x-1)$ has $m-1$ zeros inside the unit circle.
It follows that $\lambda$ is a Pisot number.

Let $g(x):=x^m-bx^{m-1}+1$.
It is easy to check $g(b-1)<0$ and $g(x)\geq g(b)=1$ for all $x\geq b$. By the intermediate value theorem the largest real root $\eta \in(b-1, b)$.
Write $g_1(x) := b x^{m-1} - 1$ and $g_2(x) := x^m$.
Note by $b\ge 3$ that $$|g_1(z)|= |bz^{m-1}-1|\geq b-1 > 1=|g_2(z)| \quad \text{for any} ~ z \in \mathbb{C} ~\text{with}~|z|=1.$$
By Rouche's theorem, $g_1(x)$ and $g(x) = g_2(x) - g_1(x)$ have the same number of zeros inside the unit circle.
Since $g_1(x)=b x^{m-1}-1 $ has $m-1$ zeros inside the unit circle, so is $g(x)$.
It follows that $\eta$ is a Pisot number.
Since $1<\la<b$ and $\la^m(\la-1)=(b-1)(\la^m-1)$, we have
\begin{align*}
 (\la-1)g(\la)&=(\la-1)(\la^m-b\la^{m-1}+1)=(b-1)(\la^m-1)+(\la-1)(1-b\la^{m-1})\\
 &=b\la^{m-1}-\la^m+\la-b=(\la^{m-1}-1)(b-\la)>0,
\end{align*}
which implies that $g(\la)>0$.
By the definition of $\eta$ we conclude that $\eta<\la$.
\end{proof}

For $\mathbf{v} = (v_1, v_2, \ldots, v_n) \in \R^n~\text{or}~\mathbb{C}^n$, let $\|\mathbf{v}\|:= \sum_{i=1}^n |v_i|$.
Set
\begin{equation}\label{eq:V}
  V :=\big\{ \mathbf v=(v_m,\ldots ,v_1) \in \R^m: \|\mathbf{v}\|=1,~ v_1 >0,~ (b-1) v_i \le v_{i+1} \le b v_i~~\forall 1\leq i< m\big\}.
\end{equation}
Let
\begin{equation}\label{eq:A-B}
  A :=
   \left(
     \begin{array}{ccccc}
    b-1 &1  & 0 & \cdots & 0 \\
    b-1 & 0 & 1 & \cdots & 0 \\
    \vdots & \vdots & \vdots & \ddots & \vdots \\
    b-1 & 0 &0 & \cdots  & 1 \\
    b-1&0 &0 &\cdots &0
     \end{array}
   \right)_{m \times m},\quad
  B :=
   \left(
     \begin{array}{ccccc}
    b &1  & 0 & \cdots & 0 \\
    0 & 0 & 1 & \cdots & 0 \\
    \vdots & \vdots & \vdots & \ddots & \vdots \\
    0 & 0 &0 & \cdots  & 1 \\
    -1&0 &0 &\cdots &0
     \end{array}
   \right)_{m \times m}.
\end{equation}
It's easy to check that the characteristic polynomials of $A$ and $B$ are $f(x) =x^m-(b-1)(x^{m-1}+x^{m-2}+\cdots+x+1)$ and $g(x) =x^m-bx^{m-1}+1$, respectively. So, by Lemma \ref{lem:pisot-number} it follows that $\la$ and $\eta$ are the spectral radii of $A$ and $B$, respectively.

\begin{lemma}\label{lem:bound-A-B}
There exists a constant $C_1 > 1$, depending on $m$ and $b$, such that for all $n \in \N$ and for all $\mathbf v\in V$ we have
$$C_1^{-1}\lambda^{n} \le \| \mathbf v A^n \|  \le C_1 \lambda^{n} \quad\text{and}\quad  C_1^{-1} \eta^{n} \le \| \mathbf v B^n \|  \le C_1  \eta^{n},$$
where $\lambda$ and $\eta$ are defined as in Theorem \ref{thm:H=P=B=AOP}.
\end{lemma}
\begin{proof}
  We first handle the matrix $A$. It's easy to verify that each entry in $A^m$ is positive, i.e., $A$ is a primitive nonnegative matrix (cf. \cite[Theorem 4.5.8]{Lind_Marcus_1995}). It follows from \cite[Theorem 4.5.12]{Lind_Marcus_1995} that there exists a constant $c>1$ depending on $m$ and $b$ such that for any $1\le i, j\le m$,
\begin{equation*}\label{eq:recursion-1}
  c^{-1} \la^k\le (A^k)_{i,j}\le c \la^k\quad \forall k\in\N,
\end{equation*}
where $\lambda$ is the largest real root of $x^m=(b-1)(x^{m-1}+x^{m-2}+\cdots+x+1)$.
Note that for any $\mathbf v=(v_m,\ldots ,v_1) \in V$ we have
$\delta_0 \le v_i < 1$ for all $1 \le i \le m,$
 where $\delta_0=(1+b+\cdots+b^{m-1})^{-1}>0$.
Therefore, we conclude that for $n \in \N$ and $\mathbf v\in V$,
$$ c^{-1} m^2 \delta_0 \lambda^{n} \le \| \mathbf v A^n \|  \le c m^2  \lambda^{n}.$$

Next, we focus on the matrix $B$. We first show that for any $\mathbf{v}\in V$ the sequence $\{\|\mathbf{v} B^n\|\}_{n=0}^\f$ is increasing.
Write $\mathbf{v}=(v_m, \ldots, v_1) \in V$.
For $n \in \N$, we recursively define $v_{n+m} = b v_{n+m-1} - v_{n}$.
Then we have $\mathbf{v} B^n = (v_{n+m}, \ldots, v_{n+1})$ for all $n \in \N_0$.
Note that $v_m > v_{m-1} >\cdots> v_1$. Then $v_{m+1} = b v_m - v_1= (b-1) v_m + (v_m -v_1) > v_m$. By induction, we can show that the sequence $\{v_n\}_{n=1}^\f$ is increasing. It follows that the sequence $\{\|\mathbf{v} B^n\|\}_{n=0}^\f$ is increasing.

Write $g(x) =x^m-bx^{m-1}+1$ for the characteristic polynomial of $B$.
Note that all roots of equation $g'(x) =0$ are $x=0$ and $x=b(m-1)/m >1$.
By Lemma \ref{lem:pisot-number}, $x=\eta$ is the unique root of equation $g(x)=0$ outside the unit circle.
Thus, the equations $g(x)=0$ and $g'(x)=0$ do not share a common root.
This means that the equation $g(x)=0$ has $m$ distinct roots in $\mathbb{C}$, denoted by $\eta, \eta_1,\ldots,\eta_{m-1}$.
By Lemma \ref{lem:pisot-number}, we have $|\eta_i| < 1$ for all  $1\le i \le m-1$.
Let $\mathbf{v}^0,\mathbf{v}^1,\mathbf{v}^2, \ldots, \mathbf{v}^{m-1} \in \mathbb{C}^{m}$ with $\|\mathbf{v}_i\|=1$ be the left eigenvectors corresponding to $\eta,\eta_1,\eta_2,\ldots,\eta_{m-1}$, respectively. That is, $\mathbf{v}^0 B = \eta \mathbf{v}^0$ and $\mathbf{v}^i B = \eta_i \mathbf{v}^i$ for all $1 \le i \le m-1$.
Then $\mathbf{v}^0,\mathbf{v}^1,\mathbf{v}^2, \ldots, \mathbf{v}^{m-1}$ are linearly independent in $\mathbb{C}^m$.
So, each $\mathbf{v}\in V$ can be uniquely written as
\begin{equation*}\label{eq:uniquely written}
  \mathbf{v} = c_0(\mathbf{v}) \mathbf{v}^0 + c_1(\mathbf{v})\mathbf{v}^1+\cdots +c_{m-1}(\mathbf{v}) \mathbf{v}^{m-1} \quad\text{with each}\quad c_i(\mathbf{v}) \in \mathbb{C}.
\end{equation*}
We claim that $ c_0(\mathbf{v}) \ne 0$ for any  $\mathbf{v}\in V$.

Suppose on the contrary $c_0(\mathbf{v}) = 0$ for some $\mathbf{v}=(v_m,\ldots,v_1)\in V$. Then we have
\[ \mathbf{v} B^n = c_1(\mathbf{v})\eta_1^n\mathbf{v}^1+\cdots +c_{m-1}(\mathbf{v}) \eta_{m-1}^n \mathbf{v}^{m-1} \quad \forall n\in\mathbb{N}. \]
Note that $|\eta_i| < 1$ for all  $1\le i \le m-1$. We obtain that
\[\lim_{n\to \f} \| \mathbf{v} B^n\| =0.\]
However, we have proved that $\| \mathbf{v} B^n\|\ge\|\mathbf{v}\|=1$.
This leads to a contradiction.

Note that $V$ is a compact set in $\mathbb{C}^m$, and for $0\le i \le m-1$ the coefficient $c_i(\mathbf{v})$ is continuous as a function of $\mathbf{v}$.
It follows from the claim that \[ c_0:=\inf\{ |c_0(\mathbf{v})|: \mathbf{v} \in V \} >0\quad\text{and}\quad {\widetilde{c}_0:=\sup\{ |c_0(\mathbf{v})|: \mathbf{v} \in V \}<\f.} \]
Note that $\sup\{ |c_i(\mathbf{v})|: \mathbf{v} \in V,~1 \le i \le m-1\}<\f$.
For $\mathbf{v} \in V$, we have $\mathbf{v} B^n =c_0(\mathbf{v})\eta^n \mathbf{v}^0+\sum_{i=1}^{m-1} c_i(\mathbf{v})\eta_i^n \mathbf{v}^i$. Therefore, there exists $n_0 \in \N$ such that for all $n \ge n_0$ and for all $\mathbf{v} \in V$,
\[ \frac{1}{2} c_0 \eta^n \le \|\mathbf{v} B^n\| \le \frac{3}{2} \widetilde{c}_0 \eta^n. \]
Furthermore, for $1 \le n \le n_0$ and $\mathbf{v} \in V$, we have
\[ \frac{1}{\eta^{n_0}} \eta^n \le 1 \le \|\mathbf{v} B^n\| \le \|\mathbf{v} B^{n_0}\| \le \frac{3}{2} \widetilde{c}_0 \eta^{n_0} \eta^n. \]
Hence, by taking $C_1=\max\{ \eta^{n_0}, 3\widetilde{c}_0 \eta^{n_0}/2, 2/c_0, 3\widetilde{c}_0/2, c/(m^2 \delta_0), cm^2\}$, we complete the proof.
\end{proof}

\begin{proposition}\label{prop:TD-PO}
  There exists a constant $C_2 >1$ such that  for any progressively overlapping sequence $\zeta \in (D_b^m)^{\N_0}$ and for any totally distinct sequence $\xi \in (D_b^m)^{\N_0}$, we have
  \[  C_2^{-1} \la^k \le |\Sigma_k^\zeta| \le C_2 \la^k  \quad\text{and}\quad C_2^{-1} \eta^k \le |\Sigma_k^\xi| \le C_2 \eta^k \quad \forall k \in \N,\]
where $\la$ and $\eta$ are defined as in Theorem  \ref{thm:H=P=B=AOP}.
\end{proposition}
\begin{proof}
By Lemma \ref{lem:bound-cardinality} we have
$ |\Sigma_{k+m}^\zeta|=(b-1)\sum_{j=0}^{m-1}|\Sigma_{k+j}^\zeta|$ for $ k\in\N.$
The recurrence relation can be rewritten as
$$(|\Sigma_{k+m}^\zeta|, |\Sigma_{k+m-1}^\zeta|, \ldots, |\Sigma_{k+1}^\zeta|) = (|\Sigma_{k+m-1}^\zeta|, |\Sigma_{k+m-2}^\zeta|,\ldots, |\Sigma_k^\zeta|)A,$$
where $A$ is defined in (\ref{eq:A-B}).
It follows that
\begin{equation}\label{eq:recursion}
  (|\Sigma_{k+m}^\zeta|, |\Sigma_{k+m-1}^\zeta|,\ldots, |\Sigma_{k+1}^\zeta|)=(|\Sigma_m^\zeta|, |\Sigma_{m-1}^\zeta|, \ldots, |\Sigma_1^\zeta|)A^k\quad\forall k\in\N.
\end{equation}
Write $\mathbf{v}_k = (|\Sigma_{k+m}^\zeta|, |\Sigma_{k+m-1}^\zeta|,\ldots, |\Sigma_{k+1}^\zeta|)$ for $k \in \N_0$.
By Lemma \ref{lem:b-1<b}, we have $\mathbf{v}_0/\|\mathbf{v}_0\| \in V$.
By Lemma \ref{lem:bound-A-B} and (\ref{eq:recursion}), there exists a constant $C_1>1$ such that \[ C_1^{-1}\|\mathbf{v}_0\| \lambda^k \le \|\mathbf{v}_k\| \le C_1\|\mathbf{v}_0\| \lambda^k \quad\forall k\in\N. \]
By Lemma \ref{lem:b-1<b}, we have \[ |\Sigma_{k}^\zeta| \le  |\Sigma_{k+1}^\zeta| \le \|\mathbf{v}_k\| =\sum_{j=1}^m |\Sigma_{k+j}^\zeta| \le |\Sigma_{k}^\zeta| \sum_{j=1}^{m} b^j \le  b^{m+1} |\Sigma_{k}^\zeta|. \]
Therefore, we conclude that
\[ \frac{\|\mathbf{v}_0\|}{C_1 b^{m+1}} \lambda^k  \le |\Sigma_{k}^\zeta| \le C_1\|\mathbf{v}_0\| \lambda^k\quad\forall k\in\N. \]

Using the same argument, by Lemmas \ref{lem:bound-cardinality} and \ref{lem:bound-A-B}, we can also show that
\[ \frac{\|\widetilde{\mathbf{v}}_0\|}{C_1 b^{m+1}} \eta^k  \le |\Sigma_{k}^\xi| \le C_1\|\widetilde{\mathbf{v}}_0\| \eta^k\quad\forall k\in\N, \]
where $\widetilde{\mathbf{v}}_0=(|\Sigma_m^\xi|, |\Sigma_{m-1}^\xi|, \ldots, |\Sigma_1^\xi|)$.
\end{proof}

\begin{remark}\label{rem:TD-PO}
For $n\in \mathbb{N}$, let  $\zeta \in (D_b^m)^{\N_0}$ be progressively overlapping at position-$k$ for all $1\leq k\leq n$, and let  $\xi \in (D_b^m)^{\N_0}$ be totally distinct at position-$k$ for all $1\leq k\leq n$.
Then by the proof of Proposition \ref{prop:TD-PO}, there exists a constant $C_2 >1$ such that
  \[  C_2^{-1} \la^k \le |\Sigma_k^\zeta| \le C_2 \la^k  \quad\text{and}\quad C_2^{-1} \eta^k \le |\Sigma_k^\xi| \le C_2 \eta^k \quad \forall 1\leq k \leq n+m.\]
\end{remark}

\begin{proof}[Proof of Theorem \ref{thm:H=P=B=AOP}]
It  follows directly from Propositions \ref{prop:Ahlfors-regular} and \ref{prop:TD-PO}.
\end{proof}

\section{Intermediate value property for dimensions of $K^\om$}\label{sec:intermediate-dimension}

In this section we will show the intermediate value property for Hausdorff and packing dimensions of $K^\om$, and prove Theorem \ref{thm:intermediate value}.
In the following we always let $b-1 < \eta < \la <b$ satisfy the equations
\[x^m-bx^{m-1}+1=0 \quad \text{and}\quad x^m-(b-1)(x^{m-1}+x^{m-2}+\cdots+x+1)=0,\]
respectively.

\begin{lemma}\label{lem:bound-cardinality-2}
  There exists a constant $C_2>1$ such that for any $\om \in (D_b^m)^{\mathbb{N}}$  we have
  \[ C_2^{-1} \eta^k \le |\Sigma_k^\om| \le C_2 \la^k \quad \forall k \in \N.\]
\end{lemma}
\begin{proof}
Let $\xi$ and $\zeta$ be  totally distinct and  progressively overlapping sequences in $(D_b^m)^{\N_0}$, respectively.
By Proposition \ref{prop:TD-PO}, it suffices to show that
\[ |\Sigma^\xi_k| \le |\Sigma_k^\om| \le |\Sigma_k^\zeta| \quad \forall k \in \N.  \]
Clearly, for all $1 \le k \le m$ we have {$|\Sigma^\xi_k| = |\Sigma_k^\om| = |\Sigma_k^\zeta|$}.

Assume that for some $k \in \N$ we have $|\Sigma_{j}^\om| \le |\Sigma_{j}^\zeta|$ for all $1 \le j < k+m$. Then
by Lemma \ref{lem:bound-cardinality} we obtain
$$|\Sigma_{k+m}^\om| \le  (b-1)\sum_{j=0}^{m-1}|\Sigma_{k+j}^\om| \le (b-1)\sum_{j=0}^{m-1}|\Sigma_{k+j}^\zeta|=|\Sigma_{k+m}^\zeta|.$$
By induction, we conclude that $|\Sigma_k^\om| \le |\Sigma_k^\zeta|$ for all $k \in \N$.

Note that $\eta$ satisfies the equation $x^m-bx^{m-1}+1=0$, i.e., $b-\eta^{-(m-1)} = \eta$.
Assume that for some $k \in \N$ we have $|\Sigma_{j+1}^\om|- |\Sigma_{j+1}^\xi|\ge \eta (|\Sigma_{j}^\om|- |\Sigma_{j}^\xi|) $ for all $1 \le j < k+m-1$.
Then by Lemma \ref{lem:bound-cardinality} we have
\begin{align*}
 |\Sigma_{k+m}^\om|- |\Sigma_{k+m}^\xi|&\ge  b(|\Sigma_{k+m-1}^\om|- |\Sigma_{k+m-1}^\xi|) -(|\Sigma_{k}^\om|- |\Sigma_{k}^\xi|)\\
 &\ge b(|\Sigma_{k+m-1}^\om|- |\Sigma_{k+m-1}^\xi|) -\eta^{-(m-1)}(|\Sigma_{k+m-1}^\om|- |\Sigma_{k+m-1}^\xi|)\\
 &=\eta(|\Sigma_{k+m-1}^\om|- |\Sigma_{k+m-1}^\xi|),
\end{align*}
By induction, we conclude that $|\Sigma_{k+1}^\om|- |\Sigma_{k+1}^\xi|\ge \eta (|\Sigma_{k}^\om|- |\Sigma_{k}^\xi|)$ for all $k \in \N$. Note that $\eta>0$ and $|\Sigma_{1}^\om|= |\Sigma_{1}^\xi|$. Thus we obtain that $|\Sigma_k^\om| \ge |\Sigma_k^\xi|$ for all $k \in \N$.
\end{proof}

Take $\om\in(D_b^m)^{\N_0}$. Recall from (\ref{eq:Gyn}) that for $\sc=c_1 c_2 \ldots c_k\in\Sigma_k^{\om}$ and $n\in\N$ we define
\[ \Sigma_{k+n}^\om(\sc) = \set{d_{1}d_{2}\ldots d_{k+n}\in \Sigma_{k+n}^{\om}:   d_1 d_2 \ldots d_k=\sc}. \]
For $k \ge m$ and $n \ge m$, by (\ref{eq:nov28-1}) and (\ref{eq:nov28-2}) in the proof of Lemma \ref{lem:key-estimate} we have showed that $ |\Sigma^\tau_{n-m+1}| \le |\Sigma_{k+n}^\om(\sc)| \le b^{m-1}|\Sigma^\tau_{n-m+1}|$ for some $\tau \in (D_b^m)^{\N_0}$.
By Lemma \ref{lem:bound-cardinality-2}, for $k \ge m$ and $n \ge m$ we have
\begin{equation}\label{eq:bound-c}
  C_2^{-1} \eta^{n-m+1} \le |\Sigma_{k+n}^\om(\sc)| \le C_2 b^{m-1} \la^{n-m+1}.
\end{equation}

For a bounded set $E \subset \R^n$, the Assouad   and lower dimensions of $E$ are defined by (cf.~\cite{Fraser-2021})
\begin{align*}
  \dim_A E  &:= \inf \left\{ s \geq 0 : \exists C > 0,   \rho>0,   \forall 0 < r < R < \rho,  \sup_{x\in E} N_r(B(x, R) \cap E) \leq C \left(\frac{R}{r}\right)^s \right\},\\
  \dim_L E  &:= \sup \left\{ s \geq 0 : \exists C > 0,   \rho>0,   \forall 0 < r < R < \rho,  \inf_{x\in E} N_r(B(x, R) \cap E) \geq C \left(\frac{R}{r}\right)^s \right\},
\end{align*}
where $N_r(F)$ is the smallest number of balls of diameter $r$ required to cover a bounded set $F$.

\begin{proposition}\label{prop:bound-dimension}
For any $\om \in (D_b^m)^{\mathbb{N}_0}$  we have
\begin{equation*}
  \frac{\log \eta}{\log b}  \leq \dim_L K^\om \le \dim_H K^\om\le \dim_P K^{\om} \le \dim_A K^\om \le \frac{\log \lambda}{\log b}.
\end{equation*}
\end{proposition}
\begin{proof}
Write $s_1=\log\eta/\log b$ and $s_2=\log \la / \log b$. By (\ref{eq:fractal-dimension}) we only need to prove $\dim_L K^\om\ge s_1$ and $\dim_A K^\om\le s_2$. By the definitions of $\dim_L$ and $\dim_A$
it suffices to prove that there exists a constant $c>1$ such that
\begin{equation*}\label{eq:lower-and-Assouad}
  c^{-1} \left( \frac{R}{r} \right)^{s_1}\leq N_r(B(x, R) \cap K^\om)\leq c \left( \frac{R}{r} \right)^{s_2} \quad \forall 0<r<R<b^{-m}, \ \forall x\in K^\om.
 \end{equation*}

Take $0 < r < R<b^{-m}$ and $x \in K^\om$.
Without loss of generality, we can assume $R/r \ge b^{m+1}$.
Then there exist $k_1, k_2 \in \N_{\geq m}$ with $k_2 \ge k_1 +m+1$ such that
\begin{equation}\label{eq:r-R}
  \frac{1}{{b}^{k_1+1}}\le R < \frac{1}{{b}^{k_1}} \quad \text{and} \quad  \frac{1}{{b}^{k_2+1}}\le r < \frac{1}{{b}^{k_2}}.
\end{equation}

Since $R \ge b^{-(k_1+1)}$, we can find $\sc \in \Sigma_{k_1+1}^\om$ such that $ I_\sc \cap K^\om \subset B(x, R)  \cap K^\om$.
Note that
$$\Sigma_{k_2}^\om(\sc) \subset \big\{\sd \in D_b^{k_2}:  I_\sd \cap (I_\sc \cap K^\om) \ne \emptyset \big\}.$$
By (\ref{eq:bound-c}), we obtain
 \[ N_r(B(x, R) \cap K^\om)  \ge N_{b^{-k_2}}(I_\sc \cap K^\om) \ge \frac{1}{2}|\Sigma_{k_2}^\om(\sc)| \ge \frac{1}{2C_2} \eta^{k_2-k_1 -m} .\]
 Note by (\ref{eq:r-R}) that $(R/r)^{s_1} \le (b^{k_2-k_1+1})^{s_1}= \eta^{k_2-k_1+1}$.
 Therefore, we conclude that
  \[ N_r(B(x, R) \cap K^\om) \ge  \frac{1}{2C_2 \eta^{m+1}} \left( \frac{R}{r} \right)^{s_1}. \]

  Since $R< b^{-k_1}$, the ball $B(x,R)$ intersects at most three intervals in $\{I_\sd : \sd\in \Sigma_{k_1}^\om\}$, denoted by $I_{\sc_1},I_{\sc_2},I_{\sc_3}$ with $\sc_i \in \Sigma_{k_1}^\om$ for $1\le i\le 3$.
  That is, $B(x,R) \cap K^\om \subset I_{\sc_1} \cup I_{\sc_2} \cup I_{\sc_3}$.
  Then the family of intervals $\{I_\sd : \sd \in \Sigma_{k_2+1}^\om(\sc_1)\cup \Sigma_{k_2+1}^\om(\sc_2)\cup\Sigma_{k_2+1}^\om(\sc_3)\}$ forms a $b^{-k_2-1}$-covering of $B(x,R) \cap K^\om$.
   Thus, by (\ref{eq:bound-c}) we obtain
   \begin{align*}
     N_r(B(x, R) \cap K^\om) & \le N_{b^{-k_2-1}}(B(x, R) \cap K^\om)\\
    & \le |\Sigma_{k_2+1}^\om(\sc_1)|+| \Sigma_{k_2+1}^\om(\sc_2)|+|\Sigma_{k_2+1}^\om(\sc_3)| \\
    & \le 3 C_2 b^{m-1} \la^{k_2-k_1-m+2}  .
   \end{align*}
   Note by (\ref{eq:r-R}) that $(R/r)^{s_2} \ge (b^{k_2-k_1-1})^{s_2}= \lambda^{k_2-k_1-1}$.
   It follows that
   \[ N_r(B(x, R) \cap K^\om)  \le  3 C_2 b^{m-1} \lambda^{3-m} \left( \frac{R}{r} \right)^{s_2},  \]
 completing the proof.
\end{proof}

In the following we fix two sequences $\{p_n\}_{n=1}^\f$ and $\{q_n\}_{n=1}^\f$ of positive integers. Set $p_0=q_0=0$ and
\[ \ell_n:=\sum_{i=0}^{n}(p_i+q_i) + mn \quad\text{for}~n \in \N_0. \]
Let $L\big(\{p_n\},\{q_n\} \big)$ be the set of sequences $\om \in (D_b^m)^{\N_0}$ that is progressively overlapping at position-$k$ for all $\ell_n < k \le \ell_n +p_{n+1}$ and is totally distinct at position-$k$ for all $\ell_n + p_{n+1} < k \le \ell_n +p_{n+1}+q_{n+1}$.
For $\om \in (D_b^m)^{\N_0}$ and $k\in \mathbb{N}_{0}$ write \[
\mathbf{N}^\om_{k}:=(|\Sigma^\om_{k+m}|, |\Sigma^\om_{k+m-1}|,\ldots, |\Sigma^\om_{k+1}|).
\]
Recall the set $V$ in (\ref{eq:V}) and the matrices $A$ and $B$ in (\ref{eq:A-B}).

\begin{lemma}\label{lem:control}
 There exists a constant $C_1>1$ such that for any $\om \in L\big(\{p_n\},\{q_n\} \big)$, we have
  \begin{equation}\label{eq:control-1}
     C_1^{-2n-1} \eta^{\sum_{i=0}^{n} q_i}\lambda^{k-\ell_n+\sum_{i=0}^{n} p_i}  \le \frac{\|\mathbf{N}^\om_{k}\|}{\|\mathbf{N}^\om_{0}\|} \le C_1^{2n+1} b^{nm} \eta^{\sum_{i=0}^{n} q_i}\lambda^{k-\ell_n+\sum_{i=0}^{n} p_i}
  \end{equation}for $\ell_n < k \le \ell_n + p_{n+1}$;
  and
  \begin{equation}\label{eq:control-2}
  \begin{split}
   C_1^{-2n-2} \eta^{k- (\ell_n +p_{n+1})+\sum_{i=0}^{n} q_i}\lambda^{\sum_{i=0}^{n+1} p_i}\le\frac{ \|\mathbf{N}^\om_{k}\|}{\|\mathbf{N}^\om_{0}\|} & \le C_1^{2n+2} b^{nm} \eta^{k- (\ell_n +p_{n+1})+\sum_{i=0}^{n} q_i}\lambda^{\sum_{i=0}^{n+1} p_i}
  \end{split}
  \end{equation}for $\ell_n +p_{n+1} < k \le \ell_n +p_{n+1} + q_{n+1}$.
\end{lemma}

\begin{proof}
  Fix a sequence $\om \in L\big(\{p_n\},\{q_n\} \big)$. By Lemma \ref{lem:bound-cardinality}, for $k\in \mathbb{N}$ we have
  \begin{align*}
  \mathbf{N}^\om_{k} & = \mathbf{N}^\om_{k-1} A \quad\text{if $\om$ is progressively overlapping at position-$k$}, \\
  \mathbf{N}^\om_{k} & = \mathbf{N}^\om_{k-1} B \quad\text{if $\om$ is totally distinct at position-$k$}.
  \end{align*}
  It follows that
  \begin{align*}
    \mathbf{N}^\om_{k} & = \mathbf{N}^\om_{\ell_n} A^{k-\ell_n} \quad \text{for}~~ \ell_n < k \le \ell_n +p_{n+1}, \\
    \mathbf{N}^\om_{k} & = \mathbf{N}^\om_{\ell_n+p_{n+1}} B^{k-\ell_n-p_{n+1}} \quad \text{for}~~ \ell_n +p_{n+1} < k \le \ell_n +p_{n+1} + q_{n+1}.
  \end{align*}
  By Lemma \ref{lem:b-1<b}, we have $\mathbf{N}^\om_{k}/\|\mathbf{N}^\om_{k}\| \in V$ for all $k \in \N_0$.
  By Lemma \ref{lem:bound-A-B}, we conclude that
  \begin{equation}\label{eq:iteration-1}
    C_1^{-1} \lambda^{k-\ell_n} \|\mathbf{N}^\om_{\ell_n}\| \le \|\mathbf{N}^\om_{k}\| \le C_1 \lambda^{k-\ell_n} \|\mathbf{N}^\om_{\ell_n}\| \quad \text{for}~~ \ell_n < k \le \ell_n +p_{n+1};
  \end{equation}
  and for $\ell_n +p_{n+1} < k \le \ell_n +p_{n+1} + q_{n+1} = \ell_{n+1} -m$,
  \begin{equation}\label{eq:iteration-2}
    C_1^{-1} \eta^{k-\ell_n-p_{n+1}} \|\mathbf{N}^\om_{\ell_n+p_{n+1}} \| \le \|\mathbf{N}^\om_{k}\| \le C_1 \eta^{k-\ell_n-p_{n+1}} \|\mathbf{N}^\om_{\ell_n+p_{n+1}} \|.
  \end{equation}
  By Lemma \ref{lem:b-1<b}, we have $\|\mathbf{N}^\om_{\ell_{n+1}-m}\| \le \|\mathbf{N}^\om_{\ell_{n+1}}\| \le b^m \|\mathbf{N}^\om_{\ell_{n+1}-m}\|$.
  This together with (\ref{eq:iteration-1}) and (\ref{eq:iteration-2}) implies that
  \begin{equation}\label{eq:iteration-3}
    C_1^{-2}  \eta^{q_{n+1}} \la^{p_{n+1}} \|\mathbf{N}^\om_{\ell_n}\|\le \|\mathbf{N}^\om_{\ell_{n+1}}\| \le C_1^2 b^m  \eta^{q_{n+1}} \la^{p_{n+1}} \|\mathbf{N}^\om_{\ell_n}\|.
  \end{equation}
  Thus, we obtain (\ref{eq:control-1}) by (\ref{eq:iteration-1}) and (\ref{eq:iteration-3}).
  Then (\ref{eq:control-2}) follows directly from (\ref{eq:control-1}) and (\ref{eq:iteration-2}).
\end{proof}

\begin{proposition}\label{prop:intermediate}
Given $0\leq s\leq t\leq 1$, if there exist sequences $\set{p_n}_{n=1}^\f, \set{q_n}_{n=1}^\f\subset\N$ satisfying
\begin{equation}\label{eq:limit-s-t}
  \lim_{n \to \f} \frac{\sum_{i=1}^n q_i}{ \sum_{i=1}^n (p_i+q_i) + p_{n+1}}=s, \quad \lim_{n \to \f} \frac{\sum_{i=1}^n q_i}{ \sum_{i=1}^n (p_i+q_i)} =t,
\end{equation}
and
\begin{equation}\label{eq:limit-0}
  \lim_{n\to\infty}\frac{n}{\sum_{i=1}^{n}(p_i+q_i)}= 0,
\end{equation}
 then for any $\om \in L\big(\{p_n\},\{q_n\} \big)$, we have
\[ \dim_H K^\om = t \frac{\log \eta}{\log b} + (1-t) \frac{\log \la}{\log b}  \quad\textrm{and}\quad  \dim_P K^\om = s \frac{\log \eta}{\log b} + (1-s) \frac{\log \la}{\log b}.  \]
\end{proposition}

To prove Proposition \ref{prop:intermediate}, we need the following inequalities, which can be checked directly by using $\la > \eta>1$.
\begin{lemma}\label{lem:inequality}
  For any $n_1,n_2,n \in \N$ , we have
  \[ \frac{(n_1+n) \log \eta + n_2 \log \la}{n_1+n_2+n} \le \frac{n_1 \log \eta + n_2 \log \la}{n_1+n_2} \le \frac{n_1 \log \eta + (n_2+n) \log \la}{n_1+n_2+n}.  \]
\end{lemma}

\begin{proof}[Proof of Proposition \ref{prop:intermediate}]
  By Lemma \ref{lem:b-1<b}, we have $|\Sigma_k^\om| \sum_{j=1}^{m} (b-1)^j \le \|\mathbf{N}^\om_{k}\| \le |\Sigma_k^\om| \sum_{j=1}^{m} b^j$.
  It follows from Theorem \ref{general-result} that
  \[ \dim_H K^\om = \liminf_{k \to \f} \frac{\log \|\mathbf{N}^\om_{k}\|}{k \log b} \quad \text{and} \quad \dim_P K^\om = \limsup_{k \to \f} \frac{\log \|\mathbf{N}^\om_{k}\|}{k \log b}.\]
  Write $\widetilde{\ell}_n := \sum_{i=1}^{n}(p_i+q_i)$. Note that $\ell_n=\sum_{i=1}^{n}(p_i+q_i)+mn.$ Then (\ref{eq:limit-0}) implies that $\widetilde{\ell}_n / \ell_n \to 1$ as $n \to \f$.

  In the following we always assume that $n \in \N$. By choosing an appropriate constant $C>1$, by Lemma \ref{lem:control} we have for $\ell_n < k \le \ell_n + p_{n+1}$,
  \begin{equation}\label{eq:april-27-1}
   C^{-n} \eta^{\sum_{i=0}^{n} q_i}\lambda^{k-\ell_n+\sum_{i=0}^{n} p_i} \le \|\mathbf{N}^\om_{k}\| \le C^n \eta^{\sum_{i=0}^{n} q_i}\lambda^{k-\ell_n+\sum_{i=0}^{n} p_i}; \end{equation}
  and for $\ell_n +p_{n+1} < k \le \ell_n +p_{n+1} + q_{n+1}=\ell_{n+1}-m$,
  \begin{equation}\label{eq:april-27-2}
   C^{-n} \eta^{k- (\ell_n +p_{n+1})+\sum_{i=0}^{n} q_i}\lambda^{\sum_{i=0}^{n+1} p_i} \le \|\mathbf{N}^\om_{k}\| \le C^{n} \eta^{k- (\ell_n +p_{n+1})+\sum_{i=0}^{n} q_i}\lambda^{\sum_{i=0}^{n+1} p_i}. \end{equation}
  By (\ref{eq:limit-s-t}) and (\ref{eq:limit-0}), one can easily verify that
  \[ \dim_H K^\om \le \lim_{n \to \f} \frac{\log \|\mathbf{N}^\om_{\ell_{n+1}-m}\|}{(\ell_{n+1}-m) \log b} = t \frac{\log \eta}{\log b} + (1-t) \frac{\log \la}{\log b}=:\alpha, \]
  and
  \[ \dim_P K^\om \ge \lim_{n \to \f} \frac{\log \|\mathbf{N}^\om_{\ell_n+p_{n+1}}\|}{(\ell_n+p_{n+1}) \log b} = s\frac{\log \eta}{\log b} + (1-s) \frac{\log \la}{\log b}=:\beta. \]
  It remains to show that $\alpha \le \dim_H K^\om \le \dim_P K^\om \le \beta$.

  For $\ell_n < k \le \ell_n + p_{n+1}$, by Lemma \ref{lem:inequality} and (\ref{eq:april-27-1}) we have
  \begin{align*}
    \frac{\log \|\mathbf{N}^\om_{k}\|}{k \log b} & \ge \frac{(\sum_{i=0}^{n}q_i)\log \eta + (k-\ell_n+\sum_{i=0}^{n} p_i) \log \la}{k \log b} - \frac{n\log C}{k \log b} \\
    & = \frac{(\sum_{i=0}^{n}q_i)\log \eta + (k-\ell_n+mn+\sum_{i=0}^{n} p_i) \log \la}{k \log b} - \frac{n\log C+mn\log \la}{k \log b} \\
    & \ge \frac{(\sum_{i=0}^{n}q_i)\log \eta + (\sum_{i=0}^{n} p_i) \log \la}{\widetilde{\ell}_n \log b} - \frac{n\log C+mn\log \la}{\widetilde{\ell}_n\log b}  \rightarrow \alpha \quad\textrm{as}~n\to\f;
  \end{align*}
  and
  \begin{align*}
    \frac{\log \|\mathbf{N}^\om_{k}\|}{k \log b} & \le \frac{(\sum_{i=0}^{n}q_i)\log \eta + (k-\ell_n+\sum_{i=0}^{n} p_i) \log \la}{k \log b} + \frac{n\log C}{k \log b} \\
    & \le \frac{(\sum_{i=0}^{n}q_i +mn)\log \eta + (k-\ell_n+\sum_{i=0}^{n} p_i) \log \la}{k \log b} + \frac{n\log C}{k \log b} \\
    & \le \frac{(\sum_{i=0}^{n}q_i)\log \eta + (\sum_{i=0}^{n+1} p_i) \log \la}{(\widetilde{\ell}_n+p_{n+1}) \log b} + \frac{n\log C}{\widetilde{\ell}_n \log b} \rightarrow \beta\quad\textrm{as}~n\to\f.
  \end{align*}

  For $\ell_n +p_{n+1} < k \le \ell_n +p_{n+1} + q_{n+1}$, again by Lemma \ref{lem:inequality} and (\ref{eq:april-27-2}) we have
  \begin{align*}
    \frac{\log \|\mathbf{N}^\om_{k}\|}{k \log b} & \ge \frac{\big(k-(\ell_n+p_{n+1})+\sum_{i=0}^{n}q_i \big)\log \eta + (\sum_{i=0}^{n+1} p_i) \log \la}{k \log b} -\frac{n\log C}{k \log b} \\
    & = \frac{\big(k-(\ell_n+p_{n+1})+\sum_{i=0}^{n}q_i \big)\log \eta + (mn+\sum_{i=0}^{n+1} p_i) \log \la}{k \log b} - \frac{n\log C+mn\log \la}{k \log b} \\
    & \ge \frac{(\sum_{i=0}^{n+1}q_i)\log \eta + (\sum_{i=0}^{n+1} p_i) \log \la}{\widetilde{\ell}_{n+1} \log b} -\frac{n\log C+mn\log \la}{\widetilde{\ell}_n\log b}  \;\rightarrow \alpha\quad\textrm{as}~n\to\f;
  \end{align*}
  and
  \begin{align*}
    \frac{\log \|\mathbf{N}^\om_{k}\|}{k \log b} & \le \frac{\big(k-(\ell_n+p_{n+1})+\sum_{i=0}^{n}q_i \big)\log \eta + (\sum_{i=0}^{n+1} p_i) \log \la}{k \log b} +\frac{n\log C}{k \log b} \\
    & \le \frac{\big(k-(\ell_n+p_{n+1})+mn+\sum_{i=0}^{n}q_i \big)\log \eta + (\sum_{i=0}^{n+1} p_i) \log \la}{k \log b} + \frac{n\log C}{k \log b} \\
    & \le \frac{(\sum_{i=0}^{n}q_i)\log \eta + (\sum_{i=0}^{n+1} p_i) \log \la}{(\widetilde{\ell}_{n} +p_{n+1}) \log b} +\frac{n\log C}{\widetilde{\ell}_n\log b}  \rightarrow \beta \quad \textrm{as}~n\to\f.
  \end{align*}

  For $\ell_{n+1}-m=\ell_n +p_{n+1} + q_{n+1}< k \le \ell_{n+1}$, by Lemma \ref{lem:b-1<b} we have $\|\mathbf{N}^\om_{\ell_{n+1} - m}\| \le \|\mathbf{N}^\om_{k}\| \le b^m\|\mathbf{N}^\om_{\ell_{n+1} - m}\|$, and hence, \[ \frac{\log \|\mathbf{N}^\om_{\ell_{n+1} - m}\|}{\ell_{n+1} \log b} \le \frac{\log \|\mathbf{N}^\om_{k}\|}{k \log b} \le \frac{\log \|\mathbf{N}^\om_{\ell_{n+1} - m}\| + m \log b}{(\ell_{n+1} - m) \log b},\] which yields $\lim_{n\to\f}\frac{\log\|\mathbf N^\om_k\|}{k\log b}=\alpha$.
  Therefore, we conclude that $\alpha \le \dim_H K^\om \le \dim_P K^\om \le \beta$, as desired.
\end{proof}

\begin{proposition}\label{prop:dim-L-A}
\begin{enumerate}[{\rm(i)}]
  \item If $\displaystyle\limsup_{n\to \f} p_n = \f$, then for any $\om \in L\big(\{p_n\},\{q_n\} \big)$, we have \[ \dim_A K^\om =\frac{\log \la}{\log b}.\]
 \item  If $\displaystyle \limsup_{n\to \f} q_n = \f$, then for any $\om \in L\big(\{p_n\},\{q_n\} \big)$, we have \[ \dim_L K^\om = \displaystyle\frac{\log \eta}{\log b}. \]
     \end{enumerate}
\end{proposition}
\begin{proof}
Write $s_1=\log\eta/\log b$ and $s_2=\log \la / \log b$.
By Proposition \ref{prop:bound-dimension}, we have $s_1 \le \dim_L K^\om \le \dim_A K^{\om} \le s_2$.

For (i) suppose that $\displaystyle\limsup_{n\to \f} p_n = \f$ and $\om \in L\big(\{p_n\},\{q_n\} \big)$.
To show $\dim_A K^\om \ge s_2$, it suffices to prove that for any $\ep>0$ and any $c>0$, there are $x \in K^\om$ and $R>r>0$ such that \[ N_r(B(x, R) \cap K^\om)>  c \left( \frac{R}{r} \right)^{s_2-\ep}. \]
Fix $\ep>0$ and $c>0$. Then we can find a sufficiently large $p_{n+1}$ such that
\begin{equation}\label{eq:choose-p-n}
  b^{\ep p_{n+1}} > 2 c \la^{m-1} C_2,
\end{equation}
where $C_2$ is defined as in Proposition \ref{prop:TD-PO}.
Take $x \in K^\om$, and let $R=b^{-\ell_n}$ and $r=b^{-(\ell_n +p_{n+1})}$.
Then we can find $\sc \in \Sigma_{\ell_n}^\om$ such that $ I_\sc \cap K^\om \subset B(x, R)  \cap K^\om$.
Note that
$$\Sigma_{\ell_n+p_{n+1}}^\om(\sc) \subset \big\{\sd \in D_b^{\ell_n+p_{n+1}}:  I_\sd \cap (I_\sc \cap K^\om) \ne \emptyset \big\}.$$
It follows that
\[ N_r(B(x, R) \cap K^\om) \ge N_{b^{-(\ell_n + p_{n+1})}}(I_\sc \cap K^\om) \ge \frac{1}{2}|\Sigma_{\ell_n+p_{n+1}}^\om(\sc)|.\]
Let {$\tau=\tau^0\tau^1\ldots=\om^{\ell_n+m-1} \om^{\ell_n + m} \ldots \in (D_b^m)^{\N_0}$}.
Since $\om \in L\big(\{p_n\},\{q_n\} \big)$, the sequence $\tau$ is progressively overlapping at position-$k$ for all {$1\le k\le p_{n+1}-m+1$}.
By (\ref{eq:nov28-1}) in the proof of Lemma \ref{lem:key-estimate} and {Remark \ref{rem:TD-PO}}, we have
\[ |\Sigma_{\ell_n+p_{n+1}}^\om(\sc)| \ge |\Sigma_{p_{n+1}-m+1}^\tau| \ge C_2^{-1} \la^{p_{n+1}-m+1}. \]
Thus, by (\ref{eq:choose-p-n}) we obtain
\[ N_r(B(x, R) \cap K^\om)\ge \frac{1}{2\la^{m-1} C_2} \la^{p_{n+1}}= \frac{b^{\ep p_{n+1}}}{2\la^{m-1} C_2} \left( \frac{R}{r} \right)^{s_2-\ep} > c \left( \frac{R}{r} \right)^{s_2-\ep}. \]

Next we prove (ii). Suppose that $\displaystyle\limsup_{n\to \f} q_n = \f$ and $\om \in L\big(\{p_n\},\{q_n\} \big)$.
To show $\dim_L K^\om \le s_1$, it suffices to prove that for any $\ep>0$ and any $c>0$, there are $x \in K^\om$ and $R>r>0$ such that \[ N_r(B(x, R) \cap K^\om) < c \left( \frac{R}{r} \right)^{s_1+\ep}. \]
Fix $\ep>0$ and $c>0$. Then we can find a sufficiently large $q_{n+1}$ such that
\begin{equation}\label{eq:choose-q-n}
  c \eta^{m-1} b^{\ep q_{n+1}} > 3 b^{m-1} C_2.
\end{equation}
Take $x \in K^\om$, and let $R=b^{-k_1}$ and $r=b^{-k_2}$ with $k_1 = \ell_n + p_{n+1}$ and $k_2=k_1+ q_{n+1}$.
Then we can find $\sc_1,\sc_2,\sc_3 \in \Sigma_{k_1}^\om$ such that $B(x, R)  \cap K^\om \subset I_{\sc_1} \cup I_{\sc_2} \cup I_{\sc_3}$.
So the family of intervals $\{I_\sd : \sd \in \Sigma_{k_2}^\om(\sc_1)\cup \Sigma_{k_2}^\om(\sc_2)\cup\Sigma_{k_2}^\om(\sc_3)\}$ forms a $r$-covering of $B(x,R) \cap K^\om$.
It follows that \[ N_r(B(x, R) \cap K^\om) \le |\Sigma_{k_2}^\om(\sc_1)|+| \Sigma_{k_2}^\om(\sc_2)|+|\Sigma_{k_2}^\om(\sc_3)|. \]
Let {$\tau=\tau^0\tau^1\ldots=\om^{k_1+m-1} \om^{k_1 + m} \ldots \in (D_b^m)^{\N_0}$}.
Since $\om \in L\big(\{p_n\},\{q_n\} \big)$, the sequence $\tau$ is totally distinct at position-$k$ for all {$1\le k\le q_{n+1}-m+1$}.
By (\ref{eq:nov28-2}) in the proof of Lemma \ref{lem:key-estimate} and {Remark \ref{rem:TD-PO}}, we have
\[ |\Sigma_{k_2}^\om(\sc_i)| \le b^{m-1} |\Sigma_{q_{n+1}-m+1}^\tau| \le b^{m-1} C_2 \eta^{q_{n+1}-m+1}. \]
Thus, by (\ref{eq:choose-q-n}) we obtain
\[ N_r(B(x, R) \cap K^\om) \le \frac{3 b^{m-1} C_2}{\eta^{m-1}} \eta^{q_{n+1}}= \frac{3 b^{m-1} C_2}{\eta^{m-1} b^{\ep q_{n+1}}} \left( \frac{R}{r} \right)^{s_1+\ep} < c \left( \frac{R}{r} \right)^{s_1+\ep}. \]
\end{proof}

\begin{lemma}\label{lem:dimension-L}
  We have $\dim_H L\big(\{p_n\},\{q_n\} \big) \ge \frac{1}{m}$.
\end{lemma}
\begin{proof}
  Fix any sequence $d_0 d_1 d_2 \ldots \in D_b^\N$, and we construct a sequence $\om=\om^0 \om^1 \ldots \in L\big(\{p_n\},\{q_n\} \big)$ as follows. We first set $d_{-(m-1)} = \cdots= d_{-2} = d_{-1} =0$.
  For $n\in\N_{ 0}$ and $0\le k \le p_1$ or $\ell_n + p_{n+1} + q_{n+1}< k \le \ell_{n+1} + p_{n+2}$, let $\om^k = d_{k-(m-1)} \ldots d_{k-1} d_k$. Then the sequence $\om$ is progressively overlapping at position-$k$ for all $\ell_n < k \le \ell_n + p_{n+1}$.
  For $\ell_n + p_{n+1} < k \le \ell_n + p_{n+1} + q_{n+1}$ we define $\om^k=\om_1^k\om_2^k\ldots\om_m^k$ recursively. Assume that $\om^i$  have been defined for all $0\le i < k$. Then we define $\om^k$ by letting $\om^k_m = d_k$ and choosing $\om^k_j \in D_b \setminus\{ \om^{k-m+j}_m\}$ for all $1 \le j \le m-1$. So the sequence $\om$ is totally distinct at position-$k$ for all $\ell_n + p_{n+1} < k \le \ell_n + p_{n+1} + q_{n+1}$.

  Thus, for any $\mathbf{d}=d_0 d_1 d_2 \ldots \in D_b^\N$, we can construct a sequence $\Upsilon(\mathbf{d})=\om^0 \om^1 \ldots \in L\big(\{p_n\},\{q_n\} \big)$ such that $\om^k_m = d_k$ for all $k \in \N_0$.
  We equip $D_b^\N$ with the metric
  \[ \widetilde{\varrho}(\mathbf{d},\mathbf{d}') = b^{-\inf\{k \in \N_0: d_k \ne d_k'\}} \quad\text{for }~\mathbf{d}=d_0 d_1 d_2 \ldots, ~\mathbf{d}'=d_0' d_1' d_2' \ldots \in D_b^\N. \]
  Then one can check that $\dim_H D_b^\N = 1$ under the metric $\widetilde{\varrho}$.
  Recall the metric $\varrho$ on $(D_b^m)^{\N_0}$ defined in (\ref{eq:metric-varrho}).
  We have \[ \varrho\Big( \Upsilon(\mathbf{d}), \Upsilon(\mathbf{d}')\Big) \ge \Big(\widetilde{\varrho}( \mathbf{d}, \mathbf{d}') \Big)^m \quad \forall \mathbf{d}, \mathbf{d}'\in D_b^\N.  \]
  Therefore, we conclude that $\dim_H L\big(\{p_n\},\{q_n\} \big) \ge \frac{1}{m}$.
\end{proof}

\begin{proof}[Proof of Theorem \ref{thm:intermediate value}]
(i) follows from Proposition \ref{prop:bound-dimension}. Next we prove (ii).
  For any $\frac{\log \eta}{\log b} \le \alpha \le \beta \le \frac{\log \la}{\log b}$ there exist $0\le s \le t \le 1$ such that
  \[\alpha=t \frac{\log \eta}{\log b} + (1-t) \frac{\log \la}{\log b} \quad\text{and}\quad \beta=s\frac{\log \eta}{\log b} + (1-s) \frac{\log \la}{\log b}. \]
  By Proposition \ref{prop:intermediate} and Lemma \ref{lem:dimension-L}, we only need to construct two sequences $\{p_n\}_{n=1}^\f$ and $\{q_n\}_{n=1}^\f$ satisfying (\ref{eq:limit-s-t}) and (\ref{eq:limit-0}).
  Let $\lceil x \rceil$ denote the smallest integer strictly larger than $x$.

  \textbf{Case I.} $0\le s < t \le 1$. Let $t_n = \min\{ t/(1-t), n \}$ and $s_n = \min\{t/s -1, n\}$ for $n \in \N$. In particular, for $t=1$ we have $t_n=n$; and for $s=0$ we have $s_n=n$. Choose arbitrarily $p_1 \in \N$, and let
  \[ q_n = \lceil t_n p_n \rceil\quad \text{and}\quad p_{n+1}= \Big\lceil s_n \sum_{i=1}^{n} (p_i+q_i)\Big\rceil \quad \forall n \in \N. \]
  Note first that $p_{n+1} \ge (t-s)n$ and $q_{n+1} \ge t(t-s)n$, which implies (\ref{eq:limit-0}).
  It is easy to check that
 \[ \lim_{n \to \f} \frac{\sum_{i=1}^{n} q_i}{\sum_{i=1}^{n} p_i} = \lim_{n \to \f}  t_n= \frac{t}{1-t} \quad\text{and}\quad \lim_{n \to \f} \frac{p_{n+1}}{\sum_{i=1}^{n} (p_i+q_i)} = \lim_{n \to \f}  s_n= \frac{t}{s}-1, \]
 which imply (\ref{eq:limit-s-t}).

  \textbf{Case II.} $0 \le s=t \le 1$. Let $p_{n} = \lceil (1-t)n \rceil$ and $q_n = \lceil tn \rceil$ for all $n \in \N$. One can easily verify (\ref{eq:limit-s-t}) and (\ref{eq:limit-0}).

Finally we prove (iii). In the case that $0< s < t < 1$, the sequences $\{p_n\}_{n=1}^\f$ and $\{q_n\}_{n=1}^\f$ constructed above satisfy $p_n \to \f$ and $q_n \to \f$. By Proposition \ref{prop:dim-L-A}, we have for any $\om \in L\big(\{p_n\},\{q_n\} \big)$,
  \[ \frac{\log \eta}{\log b} = \dim_L K^\om < \dim_H K^\om < \dim_P K^{\om} < \dim_A K^\om = \frac{\log \la}{\log b}. \]
  The proof is complete.
\end{proof}

\section{Regularity of $K^\omega$}\label{sec:regularity}

In this section, we investigate some special sequences $\om \in (D_b^m)^{\N_0}$ for which $\dim_H K^\om = \dim_P K^\om$. 
First, we generalize progressively overlapping and totally distinct sequences. 
For $\om \in (D_b^m)^{\N_0}$, define
\begin{align*}
  \Lambda_{\mathrm{PO}}^\om &:= \big\{ k \in \N: \om\;\text{is progressively overlapping at position-$k$} \big\}, \\
  \Lambda_{\mathrm{TD}}^\om &:= \big\{ k \in \N: \om\;\text{is totally distinct at position-$k$} \big\}.
\end{align*}
Recall the matrices $A$ and $B$ in (\ref{eq:A-B}), and the numbers $\lambda$ and $\eta$ defined in Lemma \ref{lem:pisot-number}. 
We also use the same notation $\mathbf{N}^\om_{k}=(|\Sigma^\om_{k+m}|, |\Sigma^\om_{k+m-1}|,\ldots, |\Sigma^\om_{k+1}|)$ for $\om \in (D_b^m)^{\N_0}$ and $k\in \mathbb{N}_{0}$. 

\begin{proposition}
Let $\om \in (D_b^m)^{\N_0}$.

{\rm(i)} If 
\begin{equation}\label{eq:almost-PO}
  \lim_{n \to \f} \frac{\#\big( \Lambda_{\mathrm{PO}}^\om \cap (0,n] \big)}{n} =1,
\end{equation}
then we have \[ \dim_H K^\om = \dim_P K^{\om} = \dim_A K^{\om} = \frac{\log \lambda}{\log b}. \]

{\rm(ii)} If \[ \lim_{n \to \f} \frac{\#\big( \Lambda_{\mathrm{TD}}^\om \cap (0,n] \big)}{n} =1, \]
then we have \[ \dim_L K^{\om} = \dim_H K^\om = \dim_P K^{\om} = \frac{\log \eta}{\log b}. \]
\end{proposition}
Before proving the proposition we remark that in (i) we may not have $\dim_L K^\om=\frac{\log\la}{\log b}$. Note that one can easily construct a sequence $\om\in(D_b^m)^{\N_0}$ such that (\ref{eq:almost-PO}) holds and the length of consecutive integers in $\Lambda^\om_{TD}$ is arbitrarily long. Then by the same argument as in the proof of Proposition \ref{prop:dim-L-A} (ii) we have $\dim_L K^\om=\frac{\log\eta}{\log b}<\frac{\log\la}{\log b}$. Similarly, in (ii) the equality $\dim_A K^\om=\frac{\log \eta}{\log b}$ may not hold.

\begin{proof}
  (i) If $\N\setminus \Lambda_{\mathrm{PO}}^\om$ is a finite set, then the sequence $\om$ is eventually progressively overlapping. By Remark \ref{rem:eventually-PO}, the set $K^\om$ is $\frac{\log \lambda}{\log b}$-Ahlfors regular. 
  In the following, we assume that the set $\N\setminus \Lambda_{\mathrm{PO}}^\om$ is infinite. 
  Write $\N\setminus \Lambda_{\mathrm{PO}}^\om = \big\{n_1, n_2, n_3,\ldots \big\}$ where $n_1 < n_2 < n_3 < \cdots$. Let $n_0=0$. 
  Note that $\om$ is progressively overlapping at position-$k$ for any $n_j < k < n_{j+1}$.
It follows from Lemma \ref{lem:bound-cardinality} that \[ \mathbf{N}^\om_{k} = \mathbf{N}^\om_{n_j} A^{k-n_j} \quad \forall n_j < k < n_{j+1}. \]
For $n_j \le k < n_{j+1}$, by Lemma \ref{lem:bound-A-B} we have
\begin{align*}
  \|\mathbf{N}^\om_{k}\| & \ge C_1^{-1} \lambda^{k-n_j} \|\mathbf{N}^\om_{n_j}\| \ge C_1^{-1} \lambda^{k-n_j} \|\mathbf{N}^\om_{n_j-1}\| \\
  & \ge C_1^{-2} \lambda^{k-n_{j-1}-1} \|\mathbf{N}^\om_{n_{j-1}}\| \ge C_1^{-2} \lambda^{k-n_{j-1}-1} \|\mathbf{N}^\om_{n_{j-1}-1}\| \\
  & \ge \cdots\cdots \\
  & \ge C_1^{-j-1} \lambda^{k-j} \|\mathbf{N}^\om_{0}\|.
\end{align*}
That is, \[ \frac{ \|\mathbf{N}^\om_{k}\| }{ \|\mathbf{N}^\om_{0}\| } \ge \frac{ \lambda^{\#(\Lambda_{\mathrm{PO}}^\om \cap (0,k])} }{ C_1^{1+\#( (0,k] \setminus \Lambda_{\mathrm{PO}}^\om)} } \quad \forall k \in \N. \]
By Theorem \ref{general-result} and (\ref{eq:almost-PO}), it follows  that \[ \dim_H K^\om \ge \frac{\log \lambda}{\log b}. \]
By Proposition \ref{prop:bound-dimension}, we conclude that \[ \dim_H K^\om = \dim_P K^{\om} = \dim_A K^{\om} = \frac{\log \lambda}{\log b}. \]

(ii) The proof is almost identical to that of (i), so we omit it here. 
\end{proof}

Next, we shall prove Theorem \ref{thm:L-p-q}, which is based on the following lemma.

\begin{lemma}\label{lem:matrix-A-B}
  For $p,q \in \N$, the matrix $A^pB^q$ is a primitive nonnegative matrix. 
\end{lemma}
\begin{proof}
  In order to describe the matrix $B^{k}$ for $k \in \N$, we define the sequence $\{v_n\}$ as follows: 
\[
\begin{cases}
  v_{-2m} = \cdots = v_{-m-1} = 0,\; v_{-m} = -1, \; v_{-m+1} = \cdots = v_{-1} = 0, \\
  v_n = b v_{n-1} - v_{n-m} \;\text{for}\; n \in \N_{0}.
\end{cases}
\]
Then we have $v_n = b^n$ for $0\le n \le m-1$. One can check that 
\begin{enumerate}[{\rm(i)}]
  \item the sequence $\{v_n\}_{n=0}^\f$ is positive and is strictly increasing, and \[ v_n > \max\{v_{n-1}, v_{n-2}, \ldots, v_{n-m+1}\} \quad \forall n \ge 0, \]
\item the sequence $\{v_n\}_{n=-m+1}^\f$ is nonnegative, and \[ v_n = b v_{n-1} - v_{n-m}\ge \max\{v_{n-1}, v_{n-2}, \ldots, v_{n-m+1}\} \quad \forall n \ge -m+2. \]
\end{enumerate}

Write $B^{k}=(b_{i,j}^{(k)})_{1\le i,j \le m}$ for $k \in \N$.
We claim that for $k \in \N$,
\begin{equation}\label{eq:matrix-B-k}
 B^k =
   \left(
     \begin{array}{ccccc}
    v_k & v_{k-1}  & \cdots & v_{k-m+2} & v_{k-m+1} \\
    -v_{k-m+1} & -v_{k-m} & \cdots & -v_{k-2m+3} & -v_{k-2m+2} \\
    \vdots & \vdots & \vdots & \vdots & \vdots \\
    -v_{k-2} & -v_{k-3} &\cdots & -v_{k-m}  & -v_{k-m-1} \\
    -v_{k-1} &-v_{k-2} &\cdots &-v_{k-m+1} &-v_{k-m}
     \end{array}
   \right)_{m \times m}.
\end{equation}
That is, $b_{1,j}^{(k)} = v_{k-j+1}$ for $1\le j \le m$, and $b_{i,j}^{(k)} = - v_{k+i-j-m}$ for $1< i \le m, 1 \le j \le m$.

We will prove the claim (\ref{eq:matrix-B-k}) by induction on $k$. 
First, by (\ref{eq:A-B}) the claim (\ref{eq:matrix-B-k}) holds for $k=1$. 
Then  we assume that the claim (\ref{eq:matrix-B-k}) holds for some $k \in \N$.  
Note that $B^{k+1}= B \cdot B^k$. 
By (\ref{eq:A-B}) we have 
\begin{itemize}
  \item for $1 \le j \le m$, $b_{1,j}^{(k+1)} = b b_{1,j}^{(k)} + b_{2,j}^{(k)} = b v_{k+1-j} - v_{k+2-j-m} = v_{(k+1)-j+1}$, where the last equality follows from (ii);
  \item for $1< i < m$ and $1 \le j \le m$, $b_{i,j}^{(k+1)} = b_{i+1,j}^{(k)} = - v_{k+1+i-j-m}$;
  \item for $1 \le j \le m$, $b_{m,j}^{(k+1)} = - b_{1,j}^{(k)} = - v_{k+1-j}$.
\end{itemize}
This means that the claim (\ref{eq:matrix-B-k}) also holds for $k+1$. 
By induction on $k$, we prove the claim (\ref{eq:matrix-B-k}). 

For $k \in \N$, write $C_k = A B^k = (c_{i,j}^{(k)})_{1\le i,j \le m}$. 
Then by (\ref{eq:A-B}) we have 
\begin{equation}\label{eq:matrix-C-k}
\begin{cases}
  c_{i,j}^{(k)} 
  = (b-1) v_{k-j+1} - v_{k-j+1-(m-i)} \;\text{for}\; 1\le i < m, 1\le j \le m, \\
  c_{m,j}^{(k)}
  = (b-1) v_{k-j+1} \;\text{for}\; 1 \le j \le m.
\end{cases}
\end{equation}
By (ii) and (\ref{eq:matrix-C-k}), we first have that the matrix $AB^k$ is nonnegative for all $k \in \N$. 
Thus, the matrix $A^p B^q$ is nonnegative for all $p \in \N, q \in \N$.  
For $k \ge m-1$, by (i) and (\ref{eq:matrix-C-k}) we obtain that all entries of the matrix $AB^k$ are positive. 
Thus, the matrix $A^p B^q$ is positive for $p \in \N, q \in \N_{\ge m-1}$. 

Next, fix $1 \le k < m-1$. Note that $v_{-m}=-1$. 
By (i) and (\ref{eq:matrix-C-k}) we have \[ c_{i,j}^{(k)} > 0 \quad\text{for}\; j\le k+1 \;\text{or}\;  j=k+1+i. \]
Using the fact that $c_{i,1}^{(k)} >0$ for $1 \le i \le m$, $c_{1,j}^{(k)} >0$ for $1 \le j \le k+1$, and $c_{i, k+1+i}^{(k)}>0$ for $1 \le i \le m-1 -k$, one can check that the matrix $A B^k$ is primitive.
For $\ell \in \N$, write $A^{\ell} B^k = (u_{i,j}^{(\ell)})_{1\le i,j \le m}$. 
By induction on $\ell$, we can show that $u_{i,1}^{(\ell)} >0$ for $1 \le i \le m$, $u_{1,j}^{(\ell)} >0$ for $1 \le j \le \min\{k+\ell,m\}$, and $u_{i, k+\ell +i}^{(\ell)}>0$ for $1 \le i \le \max\{1,m-\ell-k\}$.
This implies that the matrix $A^\ell B^k$ is primitive for all $\ell \in \N$. 

Therefore,  for any $p,q \in \N$ the matrix $A^pB^q$ is a primitive nonnegative matrix.
\end{proof}

\begin{proof}[Proof of Theorem \ref{thm:L-p-q}]
  Fix $p,q \in \N$. 
  By Lemma \ref{lem:matrix-A-B}, the matrix $A^pB^q$ is a primitive nonnegative matrix. 
  Let $\lambda_{p,q}$ be the largest eigenvalue of the matrix $A^p B^q$. 
  For $\om \in L_{p,q}$, by Lemma \ref{lem:bound-cardinality} we have \[ \mathbf{N}^\om_{k(p+q)} = \mathbf{N}^\om_{0} (A^p B^q)^k \quad \forall k \in \N. \]
  As we have done in Lemma \ref{lem:bound-A-B}, there exists a constant $C>1$ such that \[ C^{-1} \lambda_{p,q}^k \le \|\mathbf{N}^\om_{k(p+q)}\| \le C \lambda_{p,q}^k\quad \forall k \in \N. \]
  By Proposition \ref{prop:Ahlfors-regular}, the set $K^{\om}$ is $\frac{\log \lambda_{p,q}}{(p+q)\log b}$-Ahlfors regular. 
  
  Next, we consider $\lambda_{np,nq}$ for $n \in \N$. Note that for $\om \in L_{np,nq}$, we have  $\mathbf{N}^\om_{kn(p+q)} = \mathbf{N}^\om_{0} (A^{np} B^{nq})^k$ for $k \in \N$.
By Lemma \ref{lem:bound-A-B}, we have \[ C_1^{-2k} \lambda^{knp} \eta^{knq}\le \frac{\| \mathbf{N}^\om_{kn(p+q)} \|}{ \|\mathbf{N}^\om_{0}\| } \le C_1^{2k} \lambda^{knp} \eta^{knq}. \]
On the other hand, \[ \lim_{k \to \f} \frac{\log \| \mathbf{N}^\om_{kn(p+q)} \| }{kn(p+q)} = \frac{\log \lambda_{np,nq}}{n(p+q)}. \]
Thus, we obtain \[ \frac{p\log \lambda}{p+q} + \frac{q\log \eta}{p+q} - \frac{2\log C_1}{n(p+q)} \le \frac{\log \lambda_{np,nq}}{n(p+q)} \le \frac{p\log \lambda}{p+q} + \frac{q\log \eta}{p+q} + \frac{2\log C_1}{n(p+q)}. \]
It follows that \[ \lim_{n \to \f} \frac{\log \lambda_{np,nq}}{n(p+q)} = \frac{p\log \lambda}{p+q} + \frac{q\log \eta}{p+q}. \]
This means that the set \[ \bigg\{ \frac{\log \lambda_{p,q}}{(p+q)\log b}: p,q \in \N \bigg\} \] is dense in $[\log_b \eta, \log_b \lambda]$. 
\end{proof}

At the end of this section, we give a new pattern which also results in Ahlfors regularity. 

\begin{example}
  Let $m\in \mathbb{N}_{\geq 3}$ and $\om=\om^0 \om^1 \om^2\ldots \in (D_b^m)^{\N_0}$, where $\om^k = \om^k_1 \om^k_2 \ldots \om^k_m \in D_b^m$. 
  Suppose that for $k\in\N_{\geq m-1}$, we have 
\begin{equation}\label{eq:mix}
\begin{cases}
  \omega^k_1 \ldots \omega^k_{m-j} \ne \omega^{k-j}_{j+1} \ldots \omega^{k-j}_m\quad \forall 1 \le j \le m-2, \\
  \omega^k_1=\omega^{k-m+1}_m.
\end{cases}
\end{equation}

Note first that 
\[ \Sigma_{k+1}^\om=\Big(\Sigma_{k}^\om\times \big(D_b\backslash \big\{ \omega_{m}^{k-m+1}\big\}\big)\Big)\cup \big\{ d_1\ldots d_{k}\omega_{m}^{k-m+1}\in \Sigma_{k+1}^\om:d_1\ldots d_{k}\in \Sigma_{k}^\om \big\},\] where the union is disjoint. 
So, we obtain 
\begin{equation}\label{eq:nov13-1}
\begin{split}
|\Sigma_{k+1}^\om|-(b-1)|\Sigma_{k}^\om|
& = \#\left\{ d_1\ldots d_{k}\in  \Sigma_{k}^\om:d_1\ldots d_{k}\omega_{m}^{k-m+1}\in \Sigma_{k+1}^\om\right\} \\
&= \#\left\{ d_1\ldots d_{k}\in  \Sigma_{k}^\om:d_1\ldots d_{k}\omega_{1}^{k}\in \Sigma_{k+1}^\om\right\},
\end{split}
\end{equation}
where the last equality follows from $\omega^k_1=\omega^{k-m+1}_m$. 
On the other hand, 
\[ \Sigma_{k+m-1}^\om \times D_b=\Sigma_{k+m}^\om\cup \big\{ d_1\ldots d_{k}\omega_{1}^{k}\ldots \omega_{m}^{k}:d_1\ldots d_{k}\omega_{1}^{k}\ldots \omega_{m-1}^{k}\in \Sigma_{k+m-1}^\om\big\},
\]
where the union is disjoint. Thus, we have 
\begin{equation}\label{eq:nov13-2}
 b|\Sigma_{k+m-1}^\om|-|\Sigma_{k+m}^\om| = \# \big\{ d_1\ldots d_{k}\omega_{1}^{k}\ldots \omega_{m}^{k}:d_1\ldots d_{k}\omega_{1}^{k}\ldots \omega_{m-1}^{k}\in \Sigma_{k+m-1}^\om\big\}. 
 \end{equation}
For any $d_1\ldots d_{k} \in \Sigma_k^\om$, by (\ref{eq:mix}) we have \[ d_1\ldots d_{k}\omega_{1}^{k}\in \Sigma_{k+1}^\om\;\;\text{if and only if}\;\; d_1\ldots d_{k} \omega_{1}^{k}\ldots \omega_{m-1}^{k}\in \Sigma_{k+m-1}^\om. \]
Therefore, by (\ref{eq:nov13-1}) and (\ref{eq:nov13-2}) we obtain $|\Sigma_{k+1}^\om|-(b-1)|\Sigma_{k}^\om| = b|\Sigma_{k+m-1}^\om|-|\Sigma_{k+m}^\om|$, i.e., \[ |\Sigma_{k+m}^\om|=b|\Sigma_{k+m-1}^\om|-|\Sigma_{k+1}^\om|+(b-1)|\Sigma_{k}^\om|\quad \forall k \ge m-1. \]

Let $\gamma \in (b-1,b)$ be the root of equation $x^m - b x^{m-1} + x -(b-1)=0$. 
As we have done in Lemmas \ref{lem:pisot-number}, \ref{lem:bound-A-B} and Proposition \ref{prop:TD-PO}, we can show that $\gamma$ is a Pisot number, and there exists a constant $C>1$ such that \[ C^{-1} \gamma^k \le |\Sigma_{k}^\om| \le C \gamma^k \quad \forall k \in \N. \]
By Proposition \ref{prop:Ahlfors-regular}, the set $K^{\om}$ is $\frac{\log \gamma}{\log b}$-Ahlfors regular. 
Moreover, we have $\eta< \gamma < \lambda$. 
\end{example}

 Theorem \ref{thm:L-p-q} shows that for a countable and dense set of $s\in[\frac{\log \eta}{\log b}, \frac{\log \la}{\log b}]$ there are infinitely many $\om\in(D_b^m)^{\N_0}$ such that $K^\om$ is $s$-Ahlfors regular. We conjecture that this Ahlfors regular property should hold for all $s\in[\frac{\log \eta}{\log b}, \frac{\log \la}{\log b}]$.

\section{Applications in badly approximable numbers, non-recurrence points and joint spectral radius}\label{sec:applications}
In this section we will prove   Theorems \ref{thm:otwillappro} and \ref{thm:notrecurrence} for the  set of badly approximable numbers and non-recurrence points, respectively, and prove Theorem \ref{thm:finiteness property} that the finiteness  property  holds for the joint spectral radius of  $\set{A_{\sd}: \sd\in D_b^m}$.

\begin{lemma}\label{lem:notsharplowerbound}
Suppose that $m \in \N_{\ge 10}$. For any $\om=\om^0\om^1\ldots,\tau=\tau^0\tau^1\ldots \in(D_b^m)^{\N_0}$ we have
\[
\dim_H K^{\om,\tau}\ge  1- \frac{2}{m-2} \Big( \frac{\log m}{\log b} +1 \Big),
\]
where $$K^{\om,\tau}:=\set{ \sum_{i=1}^{\f}\frac{d_i}{b^i}:  d_{k+1}d_{k+2}\ldots d_{k+m} \in D_b^m \setminus \{\om^k, \tau^k\} \quad \forall k \in \N_0 }.$$
\end{lemma}
\begin{proof}
  We will construct a Moran subset of $K^{\om,\tau}$ to estimate the lower bound of its Hausdorff dimension.

  Choose $q,\ell\in \N$ such that $2q \le m < 2(q+1)$ and $b^{\ell-1} \le m <b^{\ell}$.
  Note that $m \ge 10$ and $b \ge 3$. We have $1 \le \ell < q$.
  For any $j \in \N_0$, noting that $b^\ell> 2q$, we can choose  a word $\sd^{(j)}=d_1^{(j)}d_2^{(j)} \ldots d_\ell^{(j)} \in D_b^\ell$ such that
 \begin{equation}\label{eq:construt-subset}
    \sd^{(j)}\notin  \bigcup_{i=0}^{q-1}\set{\om^{jq+i}_{q+1-i}\om^{jq+i}_{q+2-i}\ldots \om^{jq+i}_{q+\ell-i},~\tau^{jq+i}_{q+1-i}\tau^{jq+i}_{q+2-i}\ldots \tau^{jq+i}_{q+\ell-i}}.
 \end{equation}
 Then we define $$\Lambda :=\big\{ (c_i)_{i=1}^\f \in D_b^\N : c_{(j+1)q+1} c_{(j+1)q+2} \ldots c_{(j+1)q+\ell} = \sd^{(j)} \quad \forall j\in \N_0 \big\}.$$
 Take any $(c_i)_{i=1}^\f \in \Lambda$.
 For $k \in \N_0$, write $k = jq + i$ with $j \in \N_0$ and $i \in \{0,1,\ldots, q-1\}$. By (\ref{eq:construt-subset}) we have \[ c_{k+1}\ldots c_{k+m} = c_{k+1} \ldots c_{k+q-i} d_1^{(j)} \ldots d_\ell^{(j)} c_{k+q+\ell-i+1} \ldots c_{k+m} \not\in \{ \om^k,\tau^k \}.\]
 Then it follows that $\pi_b(\Lambda)  \subset K^{\om,\tau}$.
 Note that the set $\pi_b(\Lambda)$ is a homogeneous Moran set.
 By \cite[Theorem 2.1]{Feng-Wen-Wu-97}
 we obtain $$\dim_H K^{\om,\tau} \ge \dim_H \pi_b(\Lambda) = 1 - \frac{\ell}{q} \ge 1- \frac{2}{m-2} \Big( \frac{\log m}{\log b} +1 \Big),$$
 as desired.
\end{proof}

\begin{proof}[Proof of Theorem \ref{thm:otwillappro}]
Note that $K\left(\set{B_n}\right)=\bigcup_{n=0}^\f K_n$, where
\[
K_n:=\set{x\in[0,1): T_b^k (x)\notin B_k\quad \forall k \in \mathbb{N}_{\ge n}}.
\]
Since $K_n\subset K_{n+1}$ for all $n\in \mathbb{N}_{0}$, by the countable stability of Hausdorff dimension
we have
$$
  \dim_H K\left(\set{B_n}\right)=\sup_{n\ge 0}\dim_H K_n=\lim_{n\to\infty}\dim_H K_n.
$$
For $n\in \mathbb{N}_{0}$ let
\begin{equation*}
\widetilde{K}_n := T_b^n (K_n)=\set{y\in[0,1): T_b^k (y)\notin B_{k+n}\quad \forall k\in \mathbb{N}_{0}}.
\end{equation*}
Then $\dim_H \widetilde{K}_n=\dim_H K_n$ for any $n\in \mathbb{N}_{0}$. So,
\begin{equation}\label{eq:dec3-3}
  \dim_H K\left(\set{B_n}\right)=\lim_{n\to\infty}\dim_H \widetilde{K}_n.
\end{equation}

First we prove (i). Write $r_n = \mathrm{diam}~(B_{n})$ for $n\in \mathbb{N}_0$. Suppose that $r_{n}\to 0 $ as $n\to \infty$. Then for any $m\in\N_{\ge 10}$ we can find $n_0\in\N$ such that $r_{n_0+k} < b^{-m}$ for all $k \in \N_0$.
Take $n \ge n_0$. Then for any $k \in \N_0$ we can find two words $\sc_k, \sd_k\in D_b^m$ such that $B_{k+n} \subset I_{\sc_k} \cup I_{\sd_k}$. It follows that
 \begin{equation}\label{eq:dec3-4}
 \widetilde{K}^{\om,\tau}:=\set{ x \in [0,1) : T_b^k (x) \not\in I_{\sc_k} \cup I_{\sd_k} \quad \forall k \in \mathbb{N}_{0} } \subset \widetilde{K}_n,
 \end{equation}
where $\om=\sc_0\sc_1\sc_2\ldots,\tau=\sd_0\sd_1\sd_2\ldots\in\left(D_b^m\right)^{\N_0}$.
Note that $\widetilde{K}^{\om,\tau}$ and $K^{\om,\tau}$ differ up to a countable set.
By (\ref{eq:dec3-4}) and Lemma \ref{lem:notsharplowerbound}, we obtain that
$$ \dim_H \widetilde{K}_n \ge \dim_H \widetilde{K}^{\om,\tau} =\dim_H K^{\om,\tau} \ge  1 - \frac{2}{m-2} \Big( \frac{\log m}{\log b} +1 \Big). $$
It follows from (\ref{eq:dec3-3}) that
\[ \dim_H K\left(\set{B_n}\right) \ge 1 - \frac{2}{m-2} \Big( \frac{\log m}{\log b} +1 \Big)\quad \forall m\in\N_{\ge 10}. \]
Letting $m \to \f$, we conclude that $\dim_H K\left(\set{B_n}\right) =1$.

Next we prove (ii). Suppose  $\displaystyle\liminf_{n\to\infty} r_n=r >0$. Choose $m \in \N$ such that $2b^{-m} < r$. Then there exists $n_1\in\N$ such that  $r_{n_1+k}> 2b^{-m}$ for all  $k \in \N_0$.
Take  $n \ge n_1$.
For each $k \in \mathbb{N}_{0}$ we can choose $\sd_k \in D_b^m$ such that $I_{\sd_k}\subset B_{k+n}$.
Then we have
\begin{equation}\label{eq:widetilde{K}^om}
\widetilde K_n \subset \set{ x \in [0,1) : T_b^k (x) \not\in \mathrm{int}\big( I_{\sd_k} \big)  \quad \forall k \in \mathbb{N}_{0}} =: \widetilde{K}^\om,
\end{equation}
where $\om=\sd_0\sd_1\sd_2\ldots\in(D_b^m)^{\N_0}$.
Note that the sets $\widetilde{K}^\om$ and $K^\om$ differ up to a countable set.
So, by Theorem \ref{thm:intermediate value} we obtain
$$\dim_H \widetilde{K}_n \le \dim_H \widetilde{K}^\om =\dim_H K^\om \le \frac{\log \la}{\log b} < 1,$$
where $b-1<\lambda<b$ satisfies the equation $x^m -(b-1)(x^{m-1} + \cdots + x +1)=0$. It follows from (\ref{eq:dec3-3})
that $\dim_H K\left(\set{B_n}\right)< 1$, completing the proof.
\end{proof}

\begin{proof}[Proof of Theorem \ref{thm:notrecurrence}]
Note that $E(\phi)=\bigcup_{n=1}^\f E_n(\phi)$, where
\[
 E_n(\phi):=\set{x\in[0,1): |T_b^k (x)-x|\ge \phi(k)\quad \forall k\ge n}.
\]
It is clear that $E_n(\phi)\subset E_{n+1}(\phi)$ for all $n\in \N$. It follows that
\begin{equation}\label{eq:dimension-E}
  \dim_H E(\phi) =\sup_{n\geq 1}\dim_H E_n(\phi)= \lim_{n \to \f} \dim_H E_n(\phi).
\end{equation}

 First  we prove (i). Suppose $\phi(n)\to 0$ as $n\to\infty$. For any $m\in\N_{\geq 10}$ we can find $n_0\in\N$ such that $\phi(n)<b^{-m}$ for all $n\ge n_0$. For $n\ge \max\{m,n_0\}$ we have
   \[ E_n(\phi)\cap [0,b^{-m}]\supset \set{x\in [0,b^{-m}]: T_b^k (x) \not\in [0,2 b^{-m}) \quad \forall k\ge n} \supset \widetilde{E}_n, \]
   where
 \[ \widetilde{E}_n:=\set{\sum_{i=1}^\f\frac{c_i}{b^i}: c_1\ldots c_m=0^m, c_{k+1}\ldots c_{k+m}\notin D_b^m\setminus  \{ 0^m, 0^{m-1} 1 \}\quad \forall k\ge n}. \]
   By Lemma \ref{lem:notsharplowerbound} we obtain that
   \[ \dim_H E_n(\phi) \ge \dim_H E_n(\phi)\cap [0,b^{-m}] \ge \dim_H \widetilde{E}_n\ge \dim_H (T_b^n(\widetilde{E}_n)) \ge 1- \frac{2}{m-2} \Big( \frac{\log m}{\log b} +1 \Big).   \]
  It follows from (\ref{eq:dimension-E}) that
  \[
   \dim_H E(\phi)\ge 1- \frac{2}{m-2} \Big( \frac{\log m}{\log b} +1 \Big)\quad \forall m\in\N_{\ge 10}.
  \]
   Letting $m\to\infty$ we conclude that $\dim_H E\left(\phi\right)=1$ as desired.
   
   Next we prove (ii).  Suppose on the contrary that $\displaystyle\liminf_{n\to\infty}\phi(n)>0$. Note that $\phi(n)>0$ for all $n\in\N$.
Then $r :=\displaystyle\inf_{n \ge 1} \phi(n) >0$.
Choose $m \in \N_{\ge 10}$ such that $b^{-m} < r$.
Then for any word $\sd=d_1d_2\ldots d_m\in D_b^m$, we have for $n\in\N$,
\begin{align*}
 E_n(\phi)\cap I_\sd\subset \left\{x\in [0,1): T_b^k (x) \notin  I_\sd\quad \forall k\geq n\right\}\subset T_b^{-n}(\widetilde{K}^\om) \quad\textrm{with}\quad \om=\sd^\f\in\left(D_b^m\right)^{\N_0},
\end{align*}
where $\widetilde{K}^\om$ is defined in (\ref{eq:widetilde{K}^om}).
    By Theorem \ref{thm:intermediate value} it follows that
    \begin{equation}\label{eq:dec6-1}
       \dim_H (E_n(\phi)\cap I_\sd)\le \dim_H \widetilde{K}^\om=\dim_H K^\om\le \frac{\log \la}{\log b } ,
     \end{equation}
      where $\la\in(b-1,b)$ is the largest real root of $x^m=(b-1)(x^{m-1}+x^{m-2}+\cdots+x+1)$. Note that (\ref{eq:dec6-1}) holds for all $n$ and all $\sd\in D_b^m$. Then by (\ref{eq:dimension-E}) and using $E_n\left(\phi\right)\subset\bigcup_{\sd\in D_b^m} (E_n(\phi)\cap I_\sd)$ we conclude that
      \[
      \dim_H E\left(\phi\right)\le \frac{\log \la}{\log b }<1.
      \]
 \end{proof}

 Recall from (\ref{eq:joint spectral radius}) that the joint spectral radius of $\mathcal A=\set{A_{\mathbf d}: \mathbf d\in D_b^m}$ is given by
 \[ \widehat\rho(\mathcal A)=\lim_{n\to\f}\sup_{\om^0\om^1\ldots\om^{n-1}\in (D_b^m)^n}\|A_{\om^0}A_{\om^1}\cdots A_{\om^{n-1}}\|^{1/n}. \]

\begin{proof}[Proof of Theorem \ref{thm:finiteness property}]
By Lemmas \ref{lem:counting} and \ref{lem:bound-cardinality-2}, there exists a constant $C_2 > 1$ such that for any $\om^0\om^1\ldots\om^{n-1}\in (D_b^m)^n$, we have \[ \|A_{\om^0}A_{\om^1}\cdots A_{\om^{n-1}}\| \le C_2 \la^{n+m-1}. \]
It follows that $\widehat\rho(\mathcal A) \le \la$.

On the other hand, suppose that $\om= \om^{0} \om^{1} \om^{2} \ldots = (\om^0\om^1\ldots \om^{n-1})^\f\in \left(D_b^m\right)^{\N_0}$ is a periodic and progressively overlapping sequence.
By Lemma \ref{lem:counting} and Proposition \ref{prop:TD-PO}, we have
\[
C_2^{-1} \la^{nk+m-1}\le |\Sigma_{nk+m-1}^\om|=\|(A_{\om^0}A_{\om^1}\ldots A_{\om^{n-1}})^k\|\le C_2 \la^{nk+m-1}.
\]
This implies that
\[
\rho(A_{\om^0}A_{\om^1}\ldots A_{\om^{n-1}})^{\frac{1}{n}}=\lim_{k\to\infty}\|(A_{\om^0}A_{\om^1}\cdots A_{\om^{n-1}})^k\|^{\frac{1}{nk}}=\la=\widehat\rho(\mathcal A),
\]
completing the proof.
\end{proof}

\section{Final remarks}\label{sec:final-remarks}
At the end of this paper we make some remarks on our main results Theorems \ref{general-result}, \ref{thm:H=P=B=AOP} and \ref{thm:intermediate value}.

(i) In this paper we assumed  $b\in\N_{\ge 3}$ which excludes the natural base  $b=2$. Indeed, for $b=2$ our proof of Lemma \ref{lem:key-estimate} does not work, since in this case we have less choice for the connecting  word $u_1\ldots u_{m-1}\in D_b^{m-1}$. We believe that our main results Theorems \ref{general-result}, \ref{thm:H=P=B=AOP} and \ref{thm:intermediate value} still hold for $b=2$. 

(ii) The survivor set $K^\om$ studied in this paper involves  one moving hole. Our results can be extended to study   survivor sets with multiple moving holes for sufficiently large $b$ and $m$. More precisely, for $m, M\in\N$, let $\vec{\om}=(\om(1), \om(2),\ldots, \om(M))\in \prod_{\ell=1}^{M}(D_b^m)^{\N_0}$, where each $\om(\ell)=\om^0(\ell)\om^1(\ell)\ldots \in\left(D_b^m\right)^{\N_0}$. Then the survivor set with multiple moving holes $\vec{\om}$ is defined by
$$K^{\vec{\om}}:=\set{ \sum_{i=1}^{\f}\frac{d_i}{b^i}:  d_{k+1}d_{k+2}\ldots d_{k+m} \in D^m_b\backslash \bigcup_{\ell=1}^M\{\om^k(\ell)\} \quad \forall k \in \mathbb{N}_0 }.$$
By the same argument as in the proof of Theorem \ref{general-result} we can show that
\[
\dim_H K^{\vec{\om}}=\underline{\dim}_B K^{\vec{\om}}=\liminf_{n\to\f}\frac{\log\|A_{\vec{\om}^0}A_{\vec{\om}^1}\cdots A_{\vec{\om}^n}\|}{n\log b},
\]
and
\[
\dim_P K^{\vec{\om}}=\overline{\dim}_B K^{\vec{\om}}=\limsup_{n\to\f}\frac{\log\|A_{\vec{\om}^0}A_{\vec{\om}^1}\cdots A_{\vec{\om}^n}\|}{n\log b},
\]
where each $A_{\vec{\om}^j}$ is the adjacency matrix of the subshift of finite type in $(D_b^\N, \si)$ with the set of forbidden words $\set{\om^j(1), \om^j(2),\ldots, \om^j(M)}\subset D_b^m$. In this case, it is interesting to study the dimension equality  results as in Theorem \ref{thm:H=P=B=AOP} and the intermediate value property results as in Theorem \ref{thm:intermediate value}.

Note that if the blocks $\om^j(\ell), j\in\N_0, 1\le \ell\le M$ have different length, but these blocks have finitely many patterns, i.e., $W:=\set{\om^j(\ell): j\in\N_0, 1\le \ell\le M}$ is a finite set, then we can still get analogous results as in Theorem \ref{general-result} by possibly extending the length of short blocks to achieve the maximum block length in $W$.

(iii) Since we are working more on the symbolic space, our main results can be extended to homogeneous self-similar sets in $\R^d$ satisfying the strong separation condition. More precisely, for $b\in\N_{\ge 2}$,  $r\in (0,1/b)$ and $\mathcal D\subset D_b^d=\set{0,1,\ldots, b-1}^d$ with $|\mathcal D|\geq 3$, we define the iterated function system  $\mathcal F$ on $\R^d$ by
\begin{equation*}
  \mathcal{F} :=\left\{f_a(x)=rO(x+a): a\in \mathcal D\right\},
\end{equation*}
where $O$ is an orthogonal  matrix in $\R^d$ so that $\mathcal F$ satisfies the strong separation condition.
For $\om=\om^0\om^1\ldots \in(\mathcal D^m)^{\N_0}$, let
$$ K_{\mathcal{F}}^{\om}:=\set{ \sum_{n=1}^{\f} r^n O^n d_n:  d_{k+1}d_{k+2}\ldots d_{k+m} \in \mathcal D^m\setminus \{\om^k\} \quad \forall k\in \mathbb{N}_{ 0}}.$$
Then by the same argument as in the proof of Theorem \ref{general-result} we obtain
 $$\dim_H K_{\mathcal{F}}^{\om} = \underline{\dim}_B K_{\mathcal{F}}^{\om} =\liminf_{n\to\infty}\frac{1}{-n \log r}\log \|\tilde A_{\om^0}\tilde A_{\om^1}\cdots \tilde A_{\om^{n-1}}\|,$$
 and
 $$\dim_P K_{\mathcal{F}}^{\om} = \overline{\dim}_B K_{\mathcal{F}}^{\om} = \limsup_{n\to\infty}\frac{1}{-n \log r}\log\|\tilde A_{\om^0}\tilde A_{\om^1}\cdots \tilde A_{\om^{n-1}}\|,$$
 where $\tilde A_{\om^j}$ denotes the adjacency matrix of the shift of finite type in $(\mathcal D^\N, \si)$ with the forbidden block $\om^j\in \mathcal D^m$.
 Similarly, the results in  Theorems \ref{thm:H=P=B=AOP} and \ref{thm:intermediate value} can also be extended for $K^\om_{\mathcal F}$ with $b$ replaced by $\frac{1}{r}$.

\section*{Acknowledgements}
The authors thank the anonymous referee for many useful comments which greatly improve the paper, in particular Theorem 1.8. The authors would also thank Natalia Jurga for bringing their attention to  the finiteness conjecture of real matrices, and thank Hanfeng Li and Dejun Feng for some useful  discussions. Moreover, the authors thank Aihua Fan for his suggestion on Remark \ref{rem:dim-equal}. The first author was supported by Chongqing NSF: CQYC20220511052 and Scientific Research Innovation Capacity Support Project for Young Faculty No.~ZYGXQNISKYCXNLZCXM-P2P. The third author was supported by NSFC Nos.~12501110,~12471085 and the China Postdoctoral Science Foundation No.~2024M763857.


\end{document}